\documentclass[reqno, 11pt]{amsart}
\usepackage[utf8]{inputenc}
\usepackage{amsmath, amssymb}
\usepackage{latexsym}
\usepackage{graphicx, xcolor}
\usepackage{hyperref}
\usepackage[ruled, vlined]{algorithm2e}
\usepackage{setspace}
\usepackage{multirow}

\theoremstyle{plain}
\newtheorem{thm}{Theorem}[section]

\newtheorem{lem}[thm]{Lemma}
\newtheorem{cor}[thm]{Corollary}

\theoremstyle{definition}

\theoremstyle{remark}
\newtheorem{rem}[thm]{Remark}

\newcommand{\blue}[1]{{\color[rgb]{0.0,0.0,1.0}{#1}}}

\makeatletter
\def\@seccntformat#1{\csname the#1\endcsname.\quad}
\makeatother

\renewcommand{\theenumi}{\roman{enumi}}

\makeatletter
        
        \@addtoreset{equation}{section}
\makeatother

\title[A minimizing movements for evolving spirals]{
A minimizing movements approach for 
crystalline eikonal-curvature flows of spirals}

\author{Takeshi Ohtsuka}
\address[T.~Ohtsuka]{Faculty of Informatics, Gunma University,
4-2 Aramaki-cho, Maebashi, Gunma 371-8510, Japan}
\email{tohtsuka@gunma-u.ac.jp}

\author{Yen-Hsi Richard Tsai}
\address[Y.-H.~R.~Tsai]{Department of Mathematics
and Oden Institute for Computational Engineering and Sciences,
The University of Texas at Austin \\
Austin, TX, 78712, USA}
\email{ytsai@math.utexas.edu}


\begin{document}

\begin{abstract}    
 We propose an algorithm for evolving spiral curves 
 on a planar domain by normal velocities depending 
 on the so-called crystalline curvatures. 
 The algorithm uses a minimizing movements approach 
 and relies on a special level set method for embedding the spirals. 
 We present numerical simulations and comparisons 
 demonstrating the efficacy of the proposed numerical algorithm.
\end{abstract}

\maketitle

\section{Introduction}

In this paper, we propose a numerical method
for the evolution of spirals 
in a bounded domain $\Omega \subset \mathbb{R}^2$
by the crystalline curvature flow.
We denote the spirals at time $t$ by $\Sigma(t),$ and use 
${\bf n}$ for a continuous unit normal vector field
on $\Sigma(t)$ defining the orientation of the evolution. 
In this normal direction the speed of the spirals is formally defined by
\begin{equation}
 \label{geo mcf}
 V_\gamma = - \kappa_\gamma + f,
\end{equation}
where $V_\gamma$ and $\kappa_\gamma$ respectively denote 
the normal velocity and the crystalline curvature of $\Sigma(t)$ and $f$ is the given driving force.
%
Each spiral in $\Sigma(t)$ is attached to a stationary center
denoted by $a_1, \ldots, a_N \in \Omega$, 
and in the evolution of $\Sigma(t)$, portions of different spirals may merge.
To describe the merger of spirals during the evolution, we formulate $\Sigma (t)$ 
by the level set method developed in \cite{Ohtsuka:2003wi, OTG:2015JSC}.
In this formulation, $\Sigma (t)$ and ${\bf n}$ are given by
\begin{equation}
 \label{intro: lv-form}
 \Sigma (t) = \{ x \in \overline{W} ; \ u(t,x) - \theta (x) \equiv 0 \mod 2 \pi \mathbb{Z} \},
 \quad 
 {\bf n} = - \frac{\nabla (u - \theta)}{|\nabla (u - \theta)|}, 
\end{equation}
with an auxiliary function {$u(t,x)$} and the pre-determined function
\[
 \theta (x) = \sum_{j=1}^N m_j \arg (x-a_j),
\]
where $\arg (x)$ denotes 
the argument of $x \in \mathbb{R}^2$ and  $m_j \in \mathbb{Z} \setminus \{ 0 \}$,
$W = \Omega \setminus (\bigcup_{j=1}^N \overline{B_r (a_j)})$,
and $B_r (a)$ denotes an open disc with center $a \in \mathbb{R}^2$ and
a radius $r > 0$.
The domain $W$ is chosen so that the singularity of $\arg (x-a_j)$ is removed
from $\Omega$, and thus $\Sigma(t)$ is well-defined in \eqref{intro: lv-form},
even though $\theta$ is a multiple valued function.

Crystalline mean curvature flow is an anisotropic mean curvature flow
with a singular and nonlocal surface energy density.
We refer to the book \cite{Giga:2006} for a level set formulation
of geometric evolution equation.
A level set formulation of \eqref{intro: lv-form}
yields that
\begin{equation*}
 V_\gamma = \frac{u_t}{\gamma (\nabla (u - \theta))}, \quad
 \kappa_\gamma = - \mathrm{div} [\xi (\nabla (u-\theta))], 
\end{equation*}
where 
$\gamma = \gamma (p) \in C^2 (\mathbb{R}^2 \setminus \{ 0 \})$ 
and $\xi = (\partial_{p_1} \gamma, \partial_{p_2} \gamma)$
for $p=(p_1, p_2)$
denotes an anisotropic surface energy density
and 
the Cahn-Hoffman vector, respectively.
Hence, \eqref{geo mcf} will be represented as
\begin{equation}
 \label{intro: lv-mcf}
 u_t - \gamma (\nabla (u - \theta)) 
 \left\{ \mathrm{div} [\xi (\nabla (u-\theta))] + f \right\}
 = 0
 \quad \mbox{in} \ (0,T) \times W. 
\end{equation}
The Wulff diagram associated with $\gamma$ is defined as
\begin{align*}
 \mathcal{W}_\gamma = \{ p \in \mathbb{R}^2 ; \ \gamma^\circ (p) \le 1 \}, 
 \quad \mbox{where} \ \gamma^\circ (p) = \sup \{ p \cdot q ; \ \gamma (q) \le 1 \}.
\end{align*}
An important feature about the crystalline curvature $\kappa_\gamma$ 
is that $\kappa_\gamma \equiv 1$ on
the boundary $\partial \mathcal{W}_\gamma$ which is a convex polygon
(see \cite{Bellettini-Paolini:1996HMJ} for details).
It is therefore natural to identify 
a piecewise linear $\gamma^\circ$ by $\tilde{n}_j \in \mathbb{R}^2 \setminus \{ 0 \}$
for $j=0, \ldots, N_\gamma - 1$ , using the formula 
\[
 \gamma^\circ (p) = \max_{0 \le j \le N_\gamma - 1} \tilde{n}_j \cdot p.
\]
Consequently, 
$\gamma$ is also a piecewise linear, i.e., 
\begin{equation}
 \label{intro: piecewise linear gamma}
 \gamma (p) = \max_{0 \le j \le N_\gamma - 1} n_j \cdot p 
\end{equation}
for some $n_j \in \mathbb{R}^2 \setminus \{ 0 \}$ for $j=0, \ldots, N_\gamma - 1$.
See \cite{Rockafellar:1997}.
We point out that it is not immediately clear 
how to make sense of \eqref{intro: lv-mcf} classically, 
as it involves formally taking {the first and second} derivatives of $\gamma$.
In this paper, 
we 
propose a new numerical algorithm
to compute the solution of \eqref{intro: lv-mcf}, 
particularly with $\gamma$ of the form \eqref{intro: piecewise linear gamma},
in the sense of a weak formulation.

The variational approach by \cite{AlmgrenTaylorWang:1993SICON, Chambolle:2004}
is a powerful option to formulate our problem.
Almgren, Taylor and Wang \cite{AlmgrenTaylorWang:1993SICON}
launched an algorithm for the isotropic mean curvature flow 
by regarding it as a minimizing problem of the functional 
which consists of its perimeter and deformation.
This idea is also extended to the problem with driving force by \cite{LH:1995CVPDE},
or crystalline case by \cite{AT:1995JDG}.
Chambolle \cite{Chambolle:2004} proposed an algorithm
conbining the above idea and a level set method
using signed distance function from an interface.
This algorithm includes the minimization of the total variation of $u$, 
interpreted as minimizing the perimeter of each of $u$'s level sets,
and the $L^2$ norm, which measures the deformation from one discrete step to the next.
Therefore, one may consider a splitting algorithm for the minimization problem, 
using more efficient algorithms for the total variation 
and $L^2$ minimization separately and iteratively.
Oberman, Osher, Takei and the second author \cite{OOTT:2011CMS}
proposed an algorithm with split Bregman method \cite{GoldsteinOsher:2009SIIMS} for interface
motion by mean curvature, and extend it to the crystalline case.
The goal of this paper is to extend the algorithm by \cite{OOTT:2011CMS}
to the evolution of spirals.

We now quickly review the formulation in \cite{Chambolle:2004}. 
Let $\Omega \subset \mathbb{R}^2$ be {a} 
domain
and $\Sigma \subset \Omega$ be a given interfacial curve, 
separating $\Omega$ into two disjoint subsets.
Let 
$d_\Sigma \colon \Omega \to \mathbb{R}$ be a signed distance function to $\Sigma$
such that $d_\Sigma$ is positive inside of $\Sigma$.
The minimizer $w^*$ of the functional
\[
 E (w) = \int_\Omega \gamma (\nabla w) dx
 + \frac{1}{2h} \| w - d_\Sigma \|_{L^2}^2.
\]
formally satisfies 
\[
 - \mathrm{div} [ \xi (\nabla w^*) ]
 + \frac{w^* - d_\Sigma}{h} = 0.
\]
Note that $- \mathrm{div}[\xi (\nabla w^*)]$ denotes 
an anisotropic curvature of level sets of $w^*$
with an anisotropic energy density $\gamma$
(see \cite{Giga:2006}).
Moreover, $\Sigma$ moves to the region
of $\{x \in \Omega; \ d_\Sigma (x) < 0 \}$ when $V > 0$.
Hence, one can find that
\[
 \Sigma_h := \{ x \in \Omega ; \ w^* (x) = 0 \}
 = \{ x \in \Omega ; \ 
 d_\Sigma(x) = - h\mathrm{div} [ \xi (\nabla w^* (x)) ] \}.
\]
is the result of the evolution of $\Sigma$
by $V = -\kappa_\gamma$ in a short time interval $[0,h]$.

When we apply the above idea to our problem,
it is natural to consider the minimization problem of the form
\begin{equation}
 \label{spiral ROF: primitive ver}
 w \mapsto \int_W \gamma (\nabla (w - \theta)) dx
 - \int_W fw dx + \frac{\| w - g \|_{L^2}^2}{2h}, 
\end{equation}
where $g$ is a function relating to the signed ``difference''
from the current state of spiral $\Sigma (t) \subset \overline{W}$.
However, one cannot take $g$ as the ``signed distance
function'' because $\Sigma(t)$ is not an interfacial curve.
Indeed, a spiral curve is an open curve in general so that
it does not divide the domain into two regions.
For this problem, the level set method by \cite{Ohtsuka:2003wi, OTG:2015JSC}
overcomes this issue by considering \eqref{intro: lv-form}
instead of the level set of $u$.
Therefore, one option is to construct $g \in C(\overline{W})$ 
so that $g - \theta \equiv {M} d_\Sigma$ in 
a neighborhood of $\Sigma (t)$, where $M \ge 1$ 
is a constant which is chosen so that $g$ is continuous on 
$\overline{W}$. 
(Note that there is a preliminary report \cite{Ohtsuka:OWR2017} 
in which the first author mentions this algorithm.) 
Unfortunately, however, 
it may be necessary to choose an extremely large $M$, 
in particular when there is a center having multiple spirals, and the resulting algorithm does not seems particularly robust.
To overcome these difficulties, we choose a general level set function, 
i.e., $u_n = u(nh, x)$ to replace $g$ instead of signed distance function.
Then, we consider the minimizing problem of the form
\begin{equation}
 \label{spiral ROF: revised ver}
 w \mapsto \int_W \gamma (\nabla (w - \theta)) dx
 - \int_W fw dx + \frac{1}{2h} 
 \left\| \frac{ w - u_n}{\sqrt{\gamma (\nabla(u_n - \theta))}} \right\|_{L^2}^2 
\end{equation}
instead of \eqref{spiral ROF: primitive ver}.
If $w^*$ is the minimizer of the above, 
then it formally yields the finite difference approximation 
of \eqref{intro: lv-mcf}. 
Thus, we set $u_{n+1} = w^*$
and continue to solve the above with replacing $u_n$ to $u_{n+1}$.

Our idea leads to two crucial improvements:
our algorithm does not require computing signed distance to a curve,
and we may choose $m_j \in \mathbb{Z} \setminus \{ 0 \}$
so that $|m_j|$ is large.
In fact, if $|m_j| \gg 1$, then the spiral curves are very crowded
around the center $a_j$.
This situation is too difficult to compute by the algorithm
using the signed distance without additional adaptivity in meshing. 
In particular, the second improvement enables us to treat the bunching phenomena.
In experimental crystal growth, interlace patterns or so-called
illusory spirals and loops are reported in \cite{Verma:PhylosMag1951, Shtukenberg:2013}.
Our formulation can be applied to such situations.

We briefly remark on a special case 
when $u_n$ in \eqref{spiral ROF: revised ver}
has a ``flat'' portion, i.e., 
the set $\{ x \in \overline{W} ; \ 
\nabla (u_n (x) - \theta (x)) = 0 \}$
has an interior, which causes not only the issue of the zero division 
in \eqref{spiral ROF: revised ver},
but also the nonuniqueness of the evolution. 
Such a case is referred to as ``fattening.''
Evans and Spruck \cite{Evans-Spruck:1991}
pointed out the possibility of the development of an interior
in the level set equation for mean curvature flows.
In the set theoretic approach,
Soner \cite{Soner:1993JDE} 
gave an example in which fattening occurs.
The existing consistency and convergence results for mean curvature flows, such as threshold dynamics by Merriman, Bence and Osher \cite{MBO:1992}, 
or the singular limit of Allen--Cahn equation \cite{XChen:1992, ESS:1992, BSS:1993}, are established under the smooth evolution or the case when fattening never occur; see e.g. \cite{BG:1995SINN} for MBO algorithm or \cite{ESS:1992, BSS:1993}
for the singular limit of Allen--Cahn equation.
Indeed, these methods may not admit the fattening
since they consist of regularization and re-initialization in every iteration.
On the other hand, the solution computed by our approach may admit
the development of a flat region, which remains stationary until some other portions of the spiral curve approach it.
In \eqref{spiral ROF: revised ver}, we introduce a cut-off 
$\max \{ \gamma (\nabla (u - \theta)), \alpha \}$ with a small constant $\alpha > 0$
to avoid the zero division.
This means that our approach approximates the solution to a
level set equation of \eqref{geo mcf} with an approximation regularizing 
the normal velocity of level sets.

As a pioneering work on the subject of crystalline motion,
there is a front-tracking approach due to \cite{AngenentGurtin:1989ARMA, Taylor:1991}.
This idea has been also extended to the evolution of a single spiral
by \cite{Ishiwata:2014, IO:DCDS-B}.
This approach is very convenient to 
compute crystalline curvature flows
of a single spiral in a class of admissible curves. 
Therefore, we compare our approach and the model due to \cite{IO:DCDS-B}
as a benchmark test.
The value of our approach is in its flexibility to be applied in situations where
spirals may merge (and thus breaking the admissibility condition for the algorithm of \cite{IO:DCDS-B}).
In summary, in this paper we deal with spiral curves 
for which signed distance functions cannot be defined. 
In Section~2, we propose a new minimizing movements formulation, 
generalizing Chambolle's formulation to crystalline eikonal-curvature flows 
of spiral curves. 
We also present analytical results 
that provide a theoretical foundation for the proposed algorithm.
In Section~3,
we offer an efficient numerical algorithm 
for the proposed minimizing movements formulation. 
In Section~4, we provide a few numerical simulations and comparisons. 
Also, in Section~4, we generalize the algorithm 
to simulate motion laws involving multiple anisotropies.
We conclude in Section~5.




\section{The Proposed Formulation}
\label{proposed algorithms}

In this section, we propose a minimizing movements formulation 
for the evolution of spirals by a normal velocity of the form defined in \eqref{geo mcf}.
The spirals are embedded by the level set method developed in \cite{Ohtsuka:2003wi, OTG:2015JSC}. 


\subsection{The level set method for evolving spirals}
\label{level set method}

We first review the level set method for evolving spirals
due to \cite{Ohtsuka:2003wi, OTG:2015JSC} and its evolution equation
for an anisotropic eikonal-curvature flow.

Let $\Omega \subset \mathbb{R}^2$ be a bounded domain
with smooth boundary.
We denote the centers of spirals by
$a_1, a_2, \ldots, a_N \in \Omega$.
We remove small discs $B_r (a_j) = \{ x \in \mathbb{R}^2 ; \ |x - a_j| < r \}$
for $j=1, \ldots, N$ from $\Omega$, and 
set the domain for spiral growth to be
\begin{equation}
 \label{def:W}
  W = \Omega \setminus \bigcup_{j=1}^N \overline{B_r (a_j)}.
\end{equation}
We choose $r > 0$ satisfying
$\overline{B_r (a_i)} \cap \overline{B_r (a_j)} = \emptyset$
if $i \neq j$,
and $\overline{B_r (a_j)} \subset \Omega$
so that $\partial W$ is smooth. 
Let $m_j \in \mathbb{Z}$ be a signed number of
spirals associated with $a_j$, meaning
\begin{itemize}
 \item $|m_j|$ curves are attached to $a_j$
       as their endpoint,
 \item if $m_j > 0$ (resp. $m_j < 0$), then
       these curves rotate around $a_j$ with
       counterclockwise (resp. clockwise)
       rotation when $V_\gamma > 0$.
\end{itemize}
Our level set representation of spirals relies on 
a pre-determined function $\theta$ that is first introduced 
for a phase-field model of evolving spirals 
\cite{Kobayashi:1990th, Miura:2015CGD}.
Define
\begin{equation}\label{def:theta-mj-aj}
     \theta (x) = \sum_{j=1}^N m_j \arg (x-a_j).
\end{equation}
Let $\Sigma (t) \subset \overline{W}$ be
the evolving spirals at time $t \ge 0$,
and ${\bf n} \in \mathbb{S}^1$ be a
continuous unit normal vector field of $\Sigma (t)$
denoting the orientation of the evolution of $\Sigma (t)$.
Then, we describe $\Sigma (t)$
and ${\bf n}$ by
\begin{equation}
 \label{lv form}
  \Sigma (t) = \{ x \in \overline{W} ; \ 
  u(t,x) - \theta (x) \equiv 0 \mod 2 \pi \mathbb{Z} \},
  \quad {\bf n} = - \frac{\nabla (u - \theta)}{|\nabla (u - \theta)|}
\end{equation}
with an auxiliary function $u \colon [0,T) \times {\overline{W}} \to \mathbb{R}$
for some $T > 0$.
Note that $\theta$ should be a multi-valued function
to describe the spirals.
However, $\nabla \theta$ can be defined as a single-valued function.

Let $\gamma \colon \mathbb{R}^2 \to [0,\infty)$ be convex,
continuous, positively homogeneous of degree 1, 
and positive on $\mathbb{S}^1$.
It defines an anisotropic surface energy density.
According to \cite{Giga:2006}, 
an anisotropic curvature $\kappa_\gamma$ of $\Sigma (t)$
with \eqref{lv form} is given of the form
\begin{equation}
 \label{lv-anisotropic curvature}
 \kappa_\gamma = - \mathrm{div} [\xi (\nabla (u - \theta))]
\end{equation}
provided that $\gamma \in C^2 (\mathbb{R}^2 \setminus \{ 0 \})$,
where $\xi (p) = (\partial_{p_1} \gamma, \partial_{p_2} \gamma)$ for
$p = (p_1, p_2)$.
In fact, $\gamma (p) = |p|$ reduces the above 
to the isotropic 
curvature.
Note that \eqref{lv-anisotropic curvature} can be interpreted
as the first variation of the anisotropic energy
\[
 \int_W \gamma (\nabla (u - \theta)) dx
\]
for the  level sets of $u - \theta$.
This view point will be convenient to understand
the algorithms of \cite{AlmgrenTaylorWang:1993SICON, Chambolle:2004}.

The support function of {$\{ p \in \mathbb{R}^2 ; \ \gamma (p) \le 1 \}$}, denoted by
\[
 \gamma^\circ (p) = \sup \{ p \cdot q ; \ \gamma(q) \le 1 \},
\]
plays important roles for the anisotropic curvature flow. %
In fact, the followings hold:
\begin{itemize}
 \item If $\gamma, \gamma^\circ \in C^2 (\mathbb{R}^2 \setminus \{ 0 \})$, 
       then for $p \in \mathbb{R}^2 \setminus \{ 0 \}$ we have
  \begin{align}
	& \gamma (\nabla \gamma^\circ (p)) = \gamma^\circ (\nabla \gamma (p)) = 1, 
	\label{1st conj} \\
	& \nabla \gamma ( \nabla \gamma^\circ (p) ) = \frac{p}{\gamma^\circ (p)},\label{2nd conj}\\
	&\nabla \gamma^\circ (\nabla \gamma (p)) = \frac{p}{\gamma (p)}.
	\label{3rd conj}
       \end{align}.
 \item $\gamma^{\circ \circ} (p) = \gamma (p)$.
\end{itemize}
See \cite{Bellettini-Paolini:1996HMJ} 
and \cite{Rockafellar:1997} 
for details.

Equation \eqref{2nd conj} yields that
\[
 \kappa_\gamma = 1 \quad \mbox{on} \ \partial \mathcal{W}_\gamma, \ 
 \mbox{where} \ \mathcal{W}_\gamma = \{ p \in \mathbb{R}^2 ; \ 
 \gamma^\circ (p) \le 1 \}.
\]
%
In other words, by the analogy from the isotropic curvature,
the Wulff diagram plays the role of a unit ball 
in the anisotropic curvature flow.
Thus, we introduce the
Finsler metric
\begin{equation}
 \label{metric by gamma}
 d_{{\gamma}} 
 (x,y) = \gamma^\circ (x-y). 
\end{equation}
(Note that this metric is possibly not symmetric
since we do not assume $\gamma$ is symmetric.)
Then, \eqref{1st conj} implies 
$\gamma (\nabla d_{\gamma} (\cdot, y)) = 1$ in $\{ x; \ d_\gamma (x,y) > 0\}$.
This is a generalized result of $|\nabla d| = 1$
for isotropic distance $d(x,y) = |x-y|$.
Following this fact,
we introduce an anisotropic normal velocity $V_\gamma$
for $\Sigma (t)$, which is of the form
\begin{equation}
 \label{lv-anisotropic velocity}
  V_\gamma = \frac{u_t}{\gamma (\nabla  (u - \theta))}.
\end{equation}
This is analogous to the isotropic normal velocity
$V = u_t / |\nabla u|$ for the level set $\{ x ; \ u(t,x) = 0 \}$.
Consequently, we obtain the level set equation
of \eqref{geo mcf} as follows:
\begin{equation}
 \label{lv mcf}
 u_t - \gamma (\nabla (u - \theta))
 \left\{ \mathrm{div} [{\xi} (\nabla (u - \theta))] + f \right\}
 = 0 \quad 
 \mbox{in} \ (0,T) \times W.
\end{equation}

This paper primarily concerns the crystalline curvature flow,
i.e., the cases in which $\mathcal{W}_\gamma$ is a 
convex polygon with $N_\gamma$ sides ($N_\gamma\ge 3$).
For these cases, it is natural to write $\gamma^\circ$ as
\[
 \gamma^\circ (p) = \max_{0 \le j \le N_\gamma - 1} \tilde{n}_j \cdot p 
 \quad \mbox{with} \ 
 \tilde{n}_j = \tilde{r}_j (\cos \tilde\vartheta_j, \sin \tilde\vartheta_j)
 \in \mathbb{R}^2 \setminus \{ 0 \}.
\]
Here we assume that $\{ \tilde{n}_j ; \ j = 0, \ldots, N_\gamma - 1 \}$
satisfies the followings
so that $\mathcal{W}_\gamma$ is a convex polygon: 
\begin{enumerate}
 \renewcommand{\theenumi}{W\arabic{enumi}}
 \item \label{polygonal angle}
       $\tilde\vartheta_j \in [\tilde\vartheta_0, \tilde\vartheta_0 + 2 \pi]$ for 
       $j = 0,1,2, \ldots, N_\gamma - 1$,
 \item \label{convexity angle} 
       $\tilde\vartheta_j < \tilde\vartheta_{j+1} < \tilde\vartheta_j + \pi$
       for $j = 0, 1, \ldots, N_\gamma - 1$,
       where {$\tilde\vartheta_{N_\gamma} = \tilde\vartheta_0 + 2 \pi$},
 \item \label{minimal condi} 
       The set $\{ \tilde{n}_j ; \ j = 0, 1, 2, \ldots, N_\gamma - 1 \}$
       is minimal, i.e.,
       \[
	\gamma^\circ (p) \neq \max \{\tilde{n}_j \cdot p ; \ j \neq k \}
       \]
       for any $k \in \{ 0, \ldots, N_\gamma - 1 \}$.
\end{enumerate}
By $\gamma^{\circ \circ} = \gamma$ and conditions 
\eqref{polygonal angle}--\eqref{minimal condi}, 
there exist $\vartheta_j \in \mathbb{R}$ for $j=0, \ldots, N_\gamma - 1$
satisfying
\begin{equation}
 \label{def:gamma for polygons}
 \gamma (p) = \max_{0 \le j \le N_\gamma - 1} n_j \cdot p 
 \quad \mbox{with} \ 
 n_j = r_j (\cos \vartheta_j, \sin \vartheta_j) \in \mathbb{R}^2 \setminus \{ 0 \}. 
\end{equation}
Moreover,
$\{ n_j ; \ j=0, \ldots, N_\gamma - 1 \}$ still satisfies 
\eqref{polygonal angle}--\eqref{minimal condi}.
The typical example of $\gamma$ and $\gamma^\circ$ is
\begin{align*}
 & \gamma (p) = \| p \|_{\ell^1}, \quad 
 \gamma^\circ (p) = \| p \|_{\ell^\infty},
\end{align*}
and thus $\mathcal{W}_\gamma = [-1,1]^2$.
It is given by
\begin{align*}
 n_j = \sqrt{2} \left( \cos \frac{\pi (2j+1)}{4}, 
 \sin \frac{\pi(2j+1) }{4} \right), \quad
 \tilde{n}_j = \left( \cos \frac{\pi j}{2}, \sin \frac{\pi j}{2} \right).
\end{align*}
In general, the followings are required to $\gamma$ 
for the crystalline curvature flow:
\begin{enumerate}
 \renewcommand{\theenumi}{A\arabic{enumi}}
 \item \label{gamma: convexity}
       $\gamma \colon \mathbb{R}^2 \to [0,\infty)$
       is convex,
 \item \label{gamma: homogeneity}
       $\gamma$ is positively homogeneous of degree 1,
 \item \label{gamma: positivity}
       There exists $\Lambda_\gamma > 0$ such that
       $\Lambda_\gamma^{-1} \le \gamma \le \Lambda_\gamma$ on $\mathbb{S}^1$, 
 \item \label{gamma: singularity}
       $\mathcal{W}_\gamma = \{ p ; \ \gamma^\circ (p) \le 1 \}$
       is a convex polygon.
\end{enumerate}
Note that any $\gamma$ given by \eqref{def:gamma for polygons}
with \eqref{polygonal angle}--\eqref{minimal condi}
satisfies \eqref{gamma: convexity}--\eqref{gamma: singularity}.

\subsection{The proposed minimizing movements}
\label{sec: w/o dist}






Now, let the spiral $\Sigma \subset \overline{W}$ 
be given as
\[
 \Sigma 
 = \{ x \in \overline{W} ; \ u(x) - \theta (x) \equiv 0 \mod 2 \pi \mathbb{Z} \}
\]
with $u \in C(\overline{W})$.
Corresponding to the level set equation defined in \eqref{lv mcf}, 
we consider the minimizing the energy functional
\begin{equation}
 \label{energy: w/o dist}
  w \mapsto 
  \int_W \gamma (\nabla (w - \theta)) dx
  - \int_W fw dx + \frac{1}{2h}
  \left\| \frac{w - u}{\sqrt{\gamma (\nabla (u - \theta))}}
	    \right\|_{L^2}^2.
  %
  %
\end{equation}
%
Formally,
the minimizer $w^*$ 
satisfies
\[
 - \mathrm{div} [{\xi} (\nabla (w^* - \theta))]
 - f
 + \frac{w^* - u}{h \gamma (\nabla (u - \theta))} 
 %
 %
 = 0,
\] 
which implies
\begin{equation}
 \label{1st variation: lv form}
  w^* = u + h \gamma (\nabla(u - \theta))
  %
  %
  \left\{
   \mathrm{div} [{\xi} (\nabla (w^* - \theta))] + f
  \right\}.
\end{equation}
On the other hand, the implicit Euler scheme
of \eqref{lv mcf} 
takes the form
\begin{equation}\label{finite difference on time}
    \begin{aligned}
        & u(t+h) = u(t) + h\gamma (\nabla (u(t+h) - \theta))
 \left\{ \mathrm{div} [{\xi} (\nabla (u(t+h) - \theta)) ] + f \right\}.
 %
 %
    \end{aligned}
\end{equation}
Comparing \eqref{1st variation: lv form} and \eqref{finite difference on time},
we define 
\[
 S_h (\Sigma) =
 \{ x \in \overline{W} ; \ w^* (x) - \theta (x) \equiv 0 \mod 2 \pi \mathbb{Z} \}
\]
as the result of the evolution of $\Sigma$ for 
a short time step $h > 0$.
We can now iterate the stepping as follows:
\begin{enumerate}
 \item For given $\Sigma_n$ ($n \ge 0$)
       and $u_n \in C (\overline{W})$ satisfying
       \begin{equation*}
	\label{lv form n}
	 \Sigma_n = \{x \in \overline{W} ; \
	 u_n (x) - \theta (x) \equiv 0 \mod 2 \pi \mathbb{Z} \},	   
       \end{equation*}
       compute the minimizer {$w^*$} of 
       {\eqref{energy: w/o dist} with $u=u_n$}.
 \item Set {$u_{n+1} = w^*$ and}
       $$\Sigma_{n+1} = \{ x \in \overline{W} ; \ 
       u_{n+1} (x) - \theta (x) \equiv 0 \mod 2 \pi \mathbb{Z} \}.$$
\end{enumerate}
This algorithm can be applied even if $\gamma$ is not smooth, 
and we can then define
$\Sigma (t) = \Sigma_{n}$ when $nh \le t < (n+1)h $.

\noindent
\begin{rem}
 Formally,
 if $u - \theta$ 
 is a signed distance by \eqref{metric by gamma},
 then  
 \eqref{energy: w/o dist} is reduced to \eqref{spiral ROF: primitive ver}
 by \eqref{1st conj}.
 In other words, the first variation of 
 the term of $\| w - d_\Sigma \|_{L^2}^2$ in Chambolle's algorithm
 or $\| (w-u) / \sqrt{\gamma (\nabla (u - \theta))} \|_{L^2}^2$ 
 in \eqref{energy: w/o dist} approximates the normal velocity 
 of the evolving level set,
 where $d_\Sigma$ denotes the distance function from
 $\Sigma$.

 In other viewpoint, for Chambolle's algorithm,
 one has to choose the function $u$ 
 satisfying $\gamma (\nabla (u - \theta)) = 1$ at least in a tubular neighborhood
 of $\Sigma$. 
 This is the reason why Chambolle's algorithm requires the process
 constructing the distance function.
\end{rem}

\noindent
\begin{rem}
 The proposed algorithm 
 is not justified when 
 $\{ x \in \overline{W} ; \ \nabla (u_n - \theta) (x) = 0 \}$
 has a non-empty interior. 
 This means that the {integrand of the functional \eqref{energy: w/o dist}} 
 will have formally a division by 0.
 Unfortunately, when we consider the evolution of
 interfacial curve (i.e., $\theta \equiv 0$),
 there exists an example
 such that a level set $\{ x \in \Omega; \ u(t,x) = c \}$
 has a non-empty interior for $t > 0$
 even though $\{ x \in \Omega; \ u(0,x) = c \}$
 does not have an interior.
 Let us consider the following isotropic eikonal-curvature motion:
 \begin{align*}
  & u_t - |\nabla u| \left\{ \mathrm{div} \left( \frac{\nabla u}{|\nabla u|} \right) + 1 \right\}
  = 0 \quad \mbox{in} \ (0,T) \times \Omega, \\
  & u(0,x) = -|x_2| \quad \mbox{for} \ x=(x_1,x_2) \in \Omega
 \end{align*}
 with $\Omega = (\mathbb{R} / (2 \mathbb{Z}))^2$.
 Then, one can find that
 \[
  u(t,x) = \min\{ 0, t - |x_2| \}
 \]
 is a viscosity solution to the above problem.
 However, $\{ x \in \Omega ; \ u(t,x) = 0 \}$ has non-empty interior
 for every $t > 0$.

 In the practical stage, when fattening develops,
 we will introduce a cut-off of $\gamma (\nabla (u - \theta))$
 to avoid division by zero;
 see \S \ref{sec: existence of minimizing movement}, \ref{algorithm}.
\end{rem}

\begin{rem}\label{different metric: w/o dist}
 When we consider \eqref{geo mcf} with a more general mobility term of the form
 \[
  \beta (- {\bf n}) V_\gamma = - \kappa_\gamma + f,~~~\beta \neq 0,
 \] 
 one can set $u_{n+1}$ as
 \begin{equation*}
  u_{n+1} = u_n + \frac{w^* - u_n}{\beta (\nabla (w^* - \theta))},
 \end{equation*}
 where $w^*$ is the minimizer of \eqref{energy: w/o dist}.
\end{rem}

\subsection{Existence of minimizing movements}
\label{sec: existence of minimizing movement}

In this subsection, we prove the existence of
the sequence $u_n$ satisfying our proposed algorithms.
We shall address two essential questions on (i) the existence of a minimizer
of 
\eqref{energy: w/o dist} for $u_n$, 
and (ii) whether the obtained minimizer $u_{n+1}$ has $\nabla u_{n+1}$ again.

We first give a rigorous definition of the functional
$w \mapsto \int_W \gamma (\nabla (w - \theta)) dx$
to find the minimizer of {$E$}. 
By analogy of the total variation
\begin{equation*}
  [u]_{BV} 
  :=
  \sup \left\{ - \int_W u~\mathrm{div} \varphi dx ; \
	\varphi \in C^1_c (W; \mathbb{R}^2), \ |\varphi| \le 1 \right\}
\end{equation*}
and
$\gamma (p) = (\gamma^\circ)^\circ (p) = \sup \{ p \cdot q ; \ \gamma^\circ (q) \le 1 \}$
when $\gamma$ is convex,
let us define $J_\gamma \colon L^1 (W) \to [0, \infty]$ by
\begin{align}
 \label{aniso total variation}
 J_\gamma (w) = \sup \left\{
 - \int_W w~\mathrm{div} \varphi dx - \int_W \nabla \theta \cdot \varphi dx ; \
 \varphi \in C^1_c (W; \mathbb{R}^2), \ \gamma^\circ (\varphi) \le 1
 \right\}.
\end{align}
We obtain the following fundamental properties
by standard argument.

\begin{lem}
 \label{lem: aniso total variation}
 Assume that \eqref{gamma: convexity}--{\eqref{gamma: positivity}} hold.
 Then, the following hold.
 \begin{enumerate}
  \item \label{J: convex and positive}
	$J_\gamma$ is convex, and $J_\gamma \ge 0$.
  \item \label{J: well-defined}
	If $w \in BV(W)$, then $J_\gamma (w) < \infty$.
  \item \label{J: lower semiconti}
	$\liminf_{u \to v} J_\gamma (u) \ge J_\gamma (v)$ in $L^1 (W)$.
  \item \label{J: aniso peri}
	{If $w \in W^{1,1} (W)$, then
	$J_\gamma (w) = \int_W \gamma (\nabla (w - \theta)) dx$.}
 \end{enumerate}
\end{lem}

\par
\bigskip
\noindent
See Appendix for the proof of Lemma \ref{lem: aniso total variation}.

According to the result of Lemma \ref{lem: aniso total variation},
we now define $E(\cdot; g) \colon L^2 (W) \to \mathbb{R} \cup \{ \infty \}$.
For given $f, g \in L^2 (W)$ and $\psi \colon W \to \mathbb{R}$
such that $ \varphi / \sqrt{\psi} \in L^2 (W)$
for every $\varphi \in L^2 (W)$, define 
\begin{equation}
 \label{def: ROF energy}
  E (w;g) :=
  \left\{
   \begin{array}{ll}
    {\displaystyle
     J_\gamma (w) - \int_W f w dx + \frac{1}{2h} \left\| \frac{w-g}{\sqrt{\psi}} \right\|_{L^2}^2 }
     & \mbox{if} \ w \in {L^2 (W) \cap} BV(W), \\
    \infty
     & \mbox{otherwise}. 
   \end{array}
  \right.
\end{equation}

\begin{lem}\label{BV and L^1 bound}
 Assume that \eqref{gamma: convexity}--{\eqref{gamma: positivity}} hold,
 and there exist positive constants $\alpha$, $A$
 satisfying
 \begin{equation}
  \label{velocity cut-off}
  0 < \alpha < \psi < A \quad \mbox{on} \ W.
 \end{equation}
 Let $w \in L^2 (W) \cap BV(W)$ be such that
 $E (w;g) \le R$ for some constant $R > 0$.
 Then, there exist positive constants $C_0$ and $C_1$ 
 such that
 \begin{align}
  \| w \|_{BV} + \| w \|_{L^2} 
  \le C_0 + C_1 \eta^2
  \label{estimate of BV and L^2 norm}
 \end{align}
 for $\eta = \sqrt{hA}$,
 where $\| w \|_{BV} = \| w \|_{L^1} + [w]_{BV}$.
\end{lem}

\begin{proof}
 Note that $C_0$, $C_1$ and $C_2$ 
 in the following discussions
 denote numerical constants, 
 which will be changed by calculations.

 We first demonstrate that
 \begin{equation}
  \label{L1 and L2 bounds}
  \max \{ \| w \|_{L^1}, \| w \|_{L^2} \} \le C_0 + C_1 \eta + C_2 \eta^2.
 \end{equation}
 By straightforward calculation, we have
 \begin{align}
  E (w;g) & \ge J_\gamma (w) + \frac{1}{2 h A} \int_W \left[ (w-g)^2 - 2 h A f w \right] dx
  \nonumber \\
  & = J_\gamma (w)
  + \frac{1}{2h A} \int_W \left[ (w- (g + h A f))^2 - (g + h A f)^2 + g^2 \right] dx
  \nonumber \\
  & \ge J_\gamma (w) + \frac{1}{2h A} \left( \| w - G \|_{L^2}^2
  - \| G \|_{L^2}^2 \right),
  \label{lower bound of ROF}
 \end{align}
 where $G = g + h A f$.
 Then, we first obtain
 \[
 \| w - G \|_{L^2} \le \sqrt{2 h A R + \| G \|_{L^2}^2}
  \le \sqrt{2 h A R} + \| G \|_{L^2}
 \]
 by $E (w;g) \le R$ and $J_\gamma (w) \ge 0$.
 Hence, we obtain
 \begin{align*}
  \| w \|_{L^2} & \le \sqrt{2 h A R} + 2 \| G \|_{L^2}
  \le \sqrt{2 h A R} + 2 (\| g \|_{L^2} + h A \| f \|_{L^2}) \\
  & = C_0 + C_1 \eta + C_2 \eta^2.  
 \end{align*}
 This also implies that
 \[
 \| w \|_{L^1} \le |W| \| w \|_{L^2} \le C_0 + C_1 \eta + C_2 \eta^2
 \]
 since $W$ is bounded.
 Hence, we obtain \eqref{L1 and L2 bounds}.

 For the bound of $[w]_{BV}$, one can find
 \[
  J_\gamma (w) - \int_W fw dx \le {E (w;g)} \le R.
 \]
 Let $\varphi \in C_c^1 (W; \mathbb{R}^2)$ be such that
 $|\varphi| \le 1$.
 Then, since $\tilde{\varphi} = \Lambda_\gamma^{-1} \varphi$ satisfies
 $\gamma^\circ (\tilde{\varphi}) \le 1$,
 we obtain
 \begin{align*}
  - \int_W w \mathrm{div} \varphi dx
  & = \Lambda_\gamma \left( - \int_W w \mathrm{div} \tilde{\varphi} dx \right) \\
  & \le \Lambda_\gamma \left( J_\gamma (w) + \int_W \nabla \theta \cdot \tilde{\varphi} dx \right) \\
  & \le \Lambda_\gamma \left( R + \int_W f w dx \right) + \int_W \nabla \theta \cdot \varphi 
  dx \\
  & \le {\Lambda_\gamma R + |W| \| \nabla \theta \|_\infty
  + \Lambda_\gamma \| f \|_{L^2} \| w \|_{L^2}.}
  %
 \end{align*}
 It implies
 \begin{align*}
  [w]_{BV} 
  & = C_0 + C_1 \eta + C_2 \eta^2, 
 \end{align*}
 and thus
 $\| w \|_{BV} \le C_0 + C_1 \eta + C_2 \eta^2$.
 Since $\eta \le (1 + \eta^2)/2$, we obtain \eqref{estimate of BV and L^2 norm}.
\end{proof}

\begin{thm}
 \label{existence of minimizer}
 Assume that \eqref{gamma: convexity}--\eqref{gamma: positivity} hold.
 Then, for every $f, g \in L^2 (W)$
 and $\psi \colon W \to \mathbb{R}$
 satisfying \eqref{velocity cut-off},
 there exists a unique minimizer $w^* \in L^2 (W) \cap BV(W)$
 of $E (w;g)$.
\end{thm}

\begin{proof}
 By \eqref{lower bound of ROF} we observe that
 \[
  - \infty < - \| G \|_{L^2}^2 / (2 h A) 
 \le \inf_{L^2 (W) \cap BV(W)} E (\cdot;g) \le E (0;g) < \infty.
 \]
 Thus, there exists $w_n \in L^2 (W) \cap BV(W)$ such that
 \[
  \lim_{n \to \infty} E (w_n;g) 
  = \inf_{w \in L^2 (W) \cap BV(W)} E (w;g).
 \]
 We may choose $w_n$ such that $E (w_n;g) \le E (0;g)$
 by taking $n$ large enough.
 Then, the sequence $\{ w_n \}$ 
 is bounded in $L^2 (W) \cap BV(W)$ by Lemma \ref{BV and L^1 bound}.
 This implies that there exists a subsequence of $w_n \in L^2 (W) \cap BV(W)$,
 which is still denoted by $w_n$,
 and $w^* \in L^2 (W) \cap BV(W)$ such that
 $w_n \rightharpoonup w^*$ 
 in $L^2 (W)$ and
 $\| w_n - w^* \|_{L^1} \to 0$ as $n \to \infty$;
 see \cite[\S 5.2.3]{EvansGariepy:book}.
 It implies
 \[
  E (w^*;g) \le \liminf_{n \to \infty} E (w_n;g) 
 = \inf_{w \in L^2 (W) \cap BV(W)} E (w;g)
 \]
 since $E(\cdot;g)$ is lower semicontinuous in $L^1 (W)$
 by Lemma \ref{lem: aniso total variation}\eqref{J: lower semiconti}.
 Hence, 
 $w^*$ is the minimizer of $E$.

 Uniqueness is derived from strictly convexity of $E (\cdot; g)$
 on $L^2 (W) \cap BV(W)$,
 which is obtained since
 $\| (w-g) / \sqrt{\psi} \|_{L^2}^2$ is strictly convex
 for $w \in L^2 (W) \cap BV(W)$.
\end{proof}

%
%
%
%
%
%
%
%

We next discuss the second problem, i.e., 
the minimizer $w^* \in L^2 (W) \cap BV(W)$ of $E$ has
a function $\nabla w^*$ on $W$.
Now, let $w \in BV(W)$.
According to the theory of functions for bounded variation, 
there exists a Radon measure $\nu$ 
and a $\nu$-measurable function $\sigma \colon W \to \mathbb{R}^2$
satisfying;
\begin{itemize}
 \item $|\sigma| = 1$ $\nu$-a.e. in $W$, 
 \item $\int_W w \mathrm{div} \varphi dx 
       = - \int_W \varphi \cdot \sigma d \nu$
       for every $\varphi \in C^1_c (W; \mathbb{R}^2)$,
\end{itemize}
see \cite[\S 5]{EvansGariepy:book}.
Let $\sigma = (\sigma^1, \sigma^2)$ and define 
the signed measures
\[
 \nu^i (U) = \int_U \sigma^i d \nu \quad
 (i=1,2)
 \]
for every Borel set $U \subset W$.
Then, Lebesgue's decomposition theorem implies that
there exist signed measures $\nu^i_{\textup{ac}}$ and 
$\nu^i_{\textup{s}}$
such that
\begin{itemize}
 \item $\nu^i = \nu^i_{\textup{ac}} + \nu^i_{\textup{s}}$
       for $i=1,2$.
       This decomposition is unique
       for $i=1,2$.
 \item $\nu^i_{\textup{ac}}$ is absolutely continuous,
       and $\nu^i_{\textup{s}}$ is singular
       for the Lebesgue measure $\mathcal{L}^2$
       in $\mathbb{R}^2$,
 \item there exists a function $u^i \in L^1 (W)$ such that
       \[
	\nu^i_{\textup{ac}} (U) = \int_U u^i dx
        \quad (i=1,2)
       \]
       for every Borel set $U \subset W$.
\end{itemize}
By combining the above, we observe that
\begin{equation}
 \label{decomp of BV}
 \int_W w~\mathrm{div}\varphi dx
 = - \sum_{i=1}^2
 \left( \int_W \varphi^i d \nu^i_{\textup{ac}}
 + \int_W \varphi^i d \nu^i_{\textup{s}}
 \right)
 = - \sum_{i=1}^2
 \int_W u^i \varphi^i dx
 + \Phi_{\textup{s}} (\varphi)
\end{equation}
for every $\varphi = (\varphi^1, \varphi^2) \in C^1_c (W; \mathbb{R}^2)$,
where $\Phi_{\textup{s}} (\varphi)
:= - \sum_{i=1}^2 \int_W \varphi^i d \nu^i_{\textup{s}}$.
From \eqref{decomp of BV}, notice that  $w \in W^{1,1}_{\textup{loc}} (W)$ if and only if
$w \in L_{\textup{loc}}^1 (W)$, $(u^1, u^2) \in L^1_{\textup{loc}} (W)$ 
and $\Phi_s \equiv 0$.
Hence we define the weak derivative of $w \in BV(W)$ by 
\begin{equation}
 \label{BV grad}
 \nabla w := (u^1, u^2). 
\end{equation} 

Now, 
let $u_n \in L^2 (W) \cap BV(W)$ be given for our algorithm.
For positive constants $\alpha \ll 1$ and $A > \alpha$, 
set 
\begin{equation}
 \label{BV psi_n}
 \psi 
 = \psi_n := \min \{ \max \{ \gamma (\nabla u_n - \nabla \theta), \alpha \}, A \} 
\end{equation}
for $E (\cdot, u_n)$.
Then, the minimizer $u_{n+1}$ of $E(\cdot ;u_n)$ exists, 
and also has weak derivative $\nabla u_{n+1}$ in the sense of \eqref{BV grad}.
Thus, we can continue the time-stepping iterations of our algorithm.
Consequently, 
we 
conclude this section 
by presenting the following 
result for our proposed algorithms.

\par
\bigskip
\noindent
\begin{thm}
 \label{existence of minimizing movement}
 Let $f \in L^2 (W)$, $u_0 \in L^2 (W) \cap BV(W)$.
 Let $h > 0$, and $0 < \alpha < A < \infty$ be constants.
 Then, 
 there exists a unique sequence 
 $\{ u_n \}_{n \ge 0} \subset L^2 (W) \cap BV (W)$ such that
 \begin{align*}
  & u_{n+1} = \arg \min_{L^2 (W) \cap BV(W)}
  E (\cdot; u_n)
  \\
  & \qquad \mbox{with} \ 
  \psi 
  = \min \{ \max \{ \gamma (\nabla u_n - \nabla \theta), \alpha \}, A \},
 \end{align*}
 where $\nabla u_n$ is in the sense of \eqref{BV grad}.
\end{thm}

\section{Numerics for the proposed minimizing movements}

In this section, we will describe how we use the split Bregman method 
to efficiently construct the minimizer of $E(w;g)$ 
defined in \eqref{def: ROF energy}. 
We will see that the Split Bregman algorithm, 
combined with a special shrinkage formula, 
yields a simple algorithms for crystalline curvature motion.
The procedure is then summarized as Algorithm~\ref{algo: w/o dist}.



\subsection{The Split Bregman Method}

We now review the split Bregman method
due to \cite{OOTT:2011CMS}
for the calculation of the minimizer $w^*$
of $E(w;g)$. 

The key idea to find $w^*$ is to interpret the problem
to a constraint minimization problem
by dividing the dependent variable:
find a minimizer $(w^*,d^*)$ of 
\begin{align*}
 & F(w,d) := \int_W \gamma (d - \nabla \theta) dx
 - \int_W fw dx 
 + \frac{1}{2h} \left\| \frac{w - g}{\sqrt{\psi}} \right\|_{L^2}^2 \\
 & \qquad \mbox{subject to} \ d = \nabla w. 
\end{align*}
To solve this problem, 
we apply the Bregman iteration due to \cite{Bregman:1967, Osher_etal:2005}. 
Let us introduce the Bregman distance $D^{p,q}_F$ of the form
\[
 D^{p^*,q^*}_F (w,d;\hat{w},\hat{d})
 = F(w,d) - F(\hat{w},\hat{d})
 - \langle p^*, w-\hat{w} \rangle
 - \langle q^*, d-\hat{d} \rangle
\]
for $p^* \in \partial_w F(\hat{w},\hat{d})$
and $q^* \in \partial_d F(\hat{w},\hat{d})$,
where
\begin{align*}
 \partial_w F(\hat{w},\hat{d}) & = \{ p \in H^1 (W)^* ; \ 
 F(w,\hat{d}) \ge F(\hat{w},\hat{d}) - \langle p, w-\hat{w} \rangle 
 \ \mbox{for} \ w \in H^1 (W) \}, \\
 \partial_d F(\hat{w},\hat{d}) & = \{ q \in L^2 (W;\mathbb{R}^2)^* ; \ 
 F(\hat{w},d) \ge F(\hat{w},\hat{d}) - \langle q, d-\hat{d} \rangle 
 \ \mbox{for} \ d \in L^2 (W; \mathbb{R}^2) \}.
\end{align*}
Then, we set 
\begin{align*}
 (w^{k+1}, d^{k+1}) 
 & = \arg \min_{(w,d) \in H^1 (W) \times L^2 (W;\mathbb{R}^2)}
 \left[ D^{p^k, q^k}_F (w,d;w^k,d^k) 
 + \frac{\mu}{2} \| d - \nabla w \|_{L^2}^2 \right], \\
 (w^0, d^0) & = (g, 0), \quad p^k \in \partial_w F(w^k,d^k), \
 q^k \in \partial_d F(w^k,d^k).
\end{align*}
It should hold that
$(w^*,d^*) = \lim_{k \to \infty} (w^k,d^k)$.

Next, we rephrase the above iteration.
Since $(w,d) \mapsto d - \nabla w$ is linear,
the above iteration is equivalent to the following iteration:
\begin{align}
 \label{Bregman iteration: remain form1}
 (w^{k+1},d^{k+1})
 & = \arg \min_{(w,d) \in H^1 (W) \times L^2 (W;\mathbb{R}^2)}
 \left[ 
 F(w,d) + \frac{\mu}{2} \| d - \nabla w - b^k \|_{L^2}^2 \right], \\
 \label{Bregman iteration: remain form2}
 b^{k+1} & = b^k + \nabla w^{k+1} - d^{k+1}, \\
 \label{Bregman iteration: remain form3}
 (w^0, d^0, b^0 ) & = (g, 0, 0).
\end{align}
We solve \eqref{Bregman iteration: remain form1}--\eqref{Bregman iteration: remain form3}
by the following alternate iteration:
\begin{align}
 \label{split1}
 w^k_{\ell+1} & = \arg \min_{w \in H^1 (W)}
 \left[
 F(w,d^k_\ell) + \frac{\mu}{2} \| d^k_\ell - \nabla w - b^k \|_{L^2}^2 \right], 
 \\
 \label{split2}
 d^k_{\ell+1} & = \arg \min_{d \in L^2 (W;\mathbb{R}^2)}
 \left[
 F(w^k_{\ell+1},d) + \frac{\mu}{2} \| d - \nabla w^k_{\ell+1} - b^k \|_{L^2}^2 \right],  \\
 \label{split3}
 (w^k_0,d^k_0) & = (w^k, d^k).
\end{align}
For the problem \eqref{split1}, we consider the Euler-Lagrange equation
of the functional
\[
 w \mapsto - \int_W fw dx 
 + \frac{1}{2h} \left\| \frac{w - g}{\sqrt{\psi}} \right\|_{L^2}^2
 + \frac{\mu}{2} \| d^k_\ell - \nabla w - b^k \|_{L^2}^2,
\]
and then solve the following elliptic equation
with the boundary condition
\begin{align}
 \label{inner loop eq1}
 & w - h \mu \psi \Delta w 
 = g + h \psi (f - \mu \mathrm{div} (d^k_\ell - b^k))
 & & \mbox{in} \ W, \\
 \label{inner loop bc1}
 & (d^k_\ell - \nabla w - b^k) \cdot \vec{\nu} 
 = 0
 & & \mbox{on} \ \partial W,
\end{align}
where $\vec{\nu}$ denotes the outer unit normal vector field
of $W$.
The minimizing problem \eqref{split2} is
to find the minimizer of
\[
 d \mapsto \int_W \gamma (d - \nabla \theta) dx
 + \frac{\mu}{2} \| d - \nabla w^k_{\ell+1} - b^k \|_{L^2}^2.
\]
It is solved by considering the minimizer of the integrand, i.e.,
\begin{equation}
 \label{polyhedral shrinkage problem}
 d^k_{\ell+1} (x) 
 = \arg \min_{d \in \mathbb{R}^2}
 \left[ \gamma (d - \nabla \theta (x)) 
 + \frac{\mu}{2} | d - \nabla w^k_{\ell+1} (x) - b^k (x)|^2 \right]
\end{equation}
for every $x \in \overline{W}$.
It is mentioned in the next subsection.

\bigskip
\noindent
\begin{rem}
 Note that we shall use $\psi = \gamma (\nabla (g - \theta))$
 directly for \eqref{inner loop eq1} 
 as we will mention in \S \ref{sec: discretization}.
 In this case, one finds that
 \eqref{inner loop eq1} is degenerate in an interior of the set
 where $\nabla (g - \theta) = 0$.
 However, numerically it means $w^* = g$ overthere.
 It seems to be consistent with
 the minimizing movement because of the last term of $E(w;g)$. 
\end{rem}

\subsection{Shrinkage function}

We calculate the polyhedral shrinkage function
for the second minimizing problem \eqref{polyhedral shrinkage problem}
by following \cite{OOTT:2011CMS}.
By considering $\tilde{\gamma} (p) = (1/\mu) \gamma (p)$
instead of $\gamma (p)$ if necessary,
\eqref{polyhedral shrinkage problem} is summarized as follows;
find a minimizer $x^*$ of 
\begin{equation}
 \label{minimizing function}
 x \mapsto \gamma (x - z) + \frac{1}{2} |x - y|^2
\end{equation}
for given $y, z \in \mathbb{R}^2$.

First, we consider the necessary condition 
for the minimizer $x^* = x^* (y,z)$ of \eqref{minimizing function}
without the assumption \eqref{gamma: singularity}.
Remark that $\gamma (x)$ is represented as
\[
 \gamma (x) = \sup_{p \in \mathcal{W}_\gamma} p \cdot x
\]
when $\gamma$ is convex,
where 
$\mathcal{W}_\gamma = \{ p \in \mathbb{R}^2 ; \ \gamma^\circ (p) \le 1 \}$,
and $\gamma^\circ (p) = \sup \{ p \cdot q ; \ \gamma (q) \le 1 \}$.
%
The following results are revised version of 
the formula of $x^*$ and its corollary 
due to \cite{OOTT:2011CMS}.
Since the proofs are similar as those, we omit them.
See also \cite{O:2019RIMS}.

\begin{lem}
 \label{lemma 4.1}
 Let $x^* = x^*(y,z)$ be the minimizer of \eqref{minimizing function}.
 Then,
 \[
  x^* = y - P_{\mathcal{W}_\gamma} (y-z),
 \]
 where $P_{\mathcal{W}_\gamma} \colon \mathbb{R}^2 \to \mathcal{W}_\gamma$
 {is} the projection map 
 of $x$ onto $\mathcal{W}_\gamma$, which is
 \[
  P_{\mathcal{W}_\gamma} (x) = \arg \min_{y \in \mathcal{W}_\gamma} |y-x|^2.
 \]
\end{lem}

\begin{cor} \label{null situation}
 It holds that $y-z \in \mathcal{W}_\gamma$
 if and only if $x^* (y,z) = z$.
\end{cor}

Let us consider the remaining case when $y - z \notin \mathcal{W}_\gamma$.
From here on, let $\gamma$ be a convex and piecewise linear function,
i.e.,
\[
 \gamma (p) = \max_{0 \le j \le N_\gamma - 1} n_j \cdot p,
\]
where $n_j = r_j (\cos \vartheta_j, \sin \vartheta_j)$ 
($r_j > 0$ is a constant).
We here assume that 
\eqref{polygonal angle}--\eqref{minimal condi}
hold for $\{ n_j ; \ j=0, \ldots, N_\gamma - 1 \}$.
Note that,
when \eqref{polygonal angle} and \eqref{convexity angle}
hold, then
\eqref{minimal condi} is equivalent to the
following property:
\begin{itemize}
 \item[\eqref{minimal condi}']
	      $Q_i = \{ p \neq 0 ; \ n_i \cdot p > n_j \cdot p \quad
	      \mbox{for} \ j \neq i \} \neq \emptyset$,
	      and $Q_i = R_{i,i-1} \cap R_{i,i+1}$,
	      where $R_{j,k} = \{ p \neq 0; \ n_j \cdot p > n_k \cdot p \}$,
	      and $R_{N_\gamma - 1, N_\gamma} := R_{N_\gamma - 1, 0}$.
\end{itemize}
Note that \eqref{minimal condi}' 
implies $\partial Q_i = \Xi_{i-1,i} \cup \Xi_{i,i+1} \cup \{ 0 \}$,
where $\Xi_{i,i+1} = \{ p \neq 0 ; \ n_i \cdot p = n_{i+1} \cdot p > 0 \}$.
Moreover, $\gamma$ is smooth on $Q_i$ and has singularities on 
each $\Xi_{i,i + 1}$.
By using the above notations, we now characterize $x^*$
when $y - z \notin \mathcal{W}_\gamma$.

\begin{lem} \label{necessary condi.}
 Let $\gamma (p) = \max_{0 \le j \le N_\gamma - 1} n_j \cdot p$ 
 with $n_j = r_j (\cos {\vartheta_j}, \sin {\vartheta_j})$ for $r_j > 0$
 and ${\vartheta_j} \in \mathbb{R}$,
 and assume that \eqref{polygonal angle}--\eqref{minimal condi} hold.
 Let $y, z \in \mathbb{R}^2$ satisfy 
 $y - z \notin \mathcal{W}_\gamma$.
 Set
 \begin{align}
  \label{coef subdiff: result}
   \lambda_i & = \frac{(y-z) \cdot (n_i - n_{i+1}) 
  - n_i \cdot n_{i+1} + |n_{i+1}|^2}{|n_i - n_{i+1}|^2}, \\
  \label{vec subdiff}
  \xi_i & = \lambda_i n_i + (1-\lambda_i) n_{i+1}.
 \end{align}
 Let $x^*$ be the minimizer of \eqref{minimizing function}.
 Then, the following hold.
 \begin{enumerate}
  \item If $x^* - z \in \Xi_{i,i+1}$,
	then
	\begin{equation}
	 \label{min on edge}
	 \lambda_i \in [0,1], \quad \mbox{and} \quad 
	  (y-z - \xi_i) \cdot \xi_i \ge 0.
	\end{equation}
	Moreover,
	$x^* = y - \xi_i$. 
  \item If $x^* - z \in Q_i$,
	then
	\begin{equation}
	 \label{min in facet}
	  \lambda_i > 1, \quad \mbox{and} \quad \lambda_{i-1} < 0.
	\end{equation}
	Moreover,
	$x^* = y - n_i$.
 \end{enumerate}
\end{lem}

According to Corollary \ref{null situation} and Lemma \ref{necessary condi.},
we introduce the following procedure to obtain $d^k_{j+1}$.

\bigskip
\noindent
\textbf{\underline{Scheme (Sh) to obtain $d^k_{\ell+1}$ of \eqref{split2}.}}

We do the following for every $x \in \overline{W}$. 
Set $y = \nabla w^k_{\ell+1} (x) + b^k (x)$
and $z = \nabla \theta (x)$.

\medskip
\noindent
\textbf{Step 1.
Check either $y-z \in \mathcal{W}_\gamma$
(i.e. $\gamma^\circ (y-z) \le 1$)
or not.}

If $y-z \in \mathcal{W}_\gamma$,
then $x^* = z$ by Corollary \ref{null situation}.
Otherwise, we go to the next step.

\medskip
\noindent
\textbf{Step 2.
When $y-z \notin \mathcal{W}_\gamma$,
check either $x^* - z \in \bigcup_{i=0}^{N_\gamma - 1} \Xi_{i,i+1}$
or $x^* - z \in \bigcup_{i=0}^{N_\gamma - 1} Q_i$.}

Note that Lemma \ref{necessary condi.} gives a necessary condition
for the minimizer $x^*$.
Then, we can divide the situation into the following two cases.

\medskip
\noindent
\textbf{Case A. 
There is no $i \in \{ 0, 1, 2, \ldots, N_\gamma - 1 \}$
satisfying \eqref{min on edge}.}

In this case we find $x^* - z \notin \bigcup_{i=0}^{N_\gamma - 1} \Xi_{i,i+1}$
and then $x^* - z \in \bigcup_{i=0}^{N_\gamma - 1} Q_i$.
Therefore, find $i_0 \in \{0, 1, \ldots, N_\gamma - 1\}$ satisfying
\eqref{min in facet}.
If there exists such an $i_0$ uniquely,
then set
\[
 x^* = y - n_{i_0}.
\]
If there exist several $i$ satisfying
\eqref{min in facet},
choose $i_0$ 
which minimize
\[
 i \mapsto \gamma (\tilde{x}^*_i - z) + \frac{1}{2} |\tilde{x}^*_i - y|^2
 \quad \mbox{subject to \eqref{min in facet}},
\] 
where $\tilde{x}^*_i = y - n_i$.

\medskip
\noindent
\textbf{Case B. 
There is no $i \in \{ 0, 1, 2, \ldots, N_\gamma - 1 \}$
satisfying \eqref{min in facet}.}

In this case we find $x^* - z \notin \bigcup_{i=0}^{N_\gamma - 1} Q_i$
and then $x^* - z \in \bigcup_{i=0}^{N_\gamma - 1} \Xi_{i,i+1}$.
Therefore, find $i_0 \in \{ 0, 1, \ldots, N_\gamma - 1 \}$ satisfying
\eqref{min on edge}.
If there exists such an $i_0$ uniquely, then set
\[
 x^* = y - \xi_{i_0}.
\]
If there exist several $i$ satisfying
\eqref{min on edge},
choose $i_0$ as which minimize
\[
 i \mapsto \gamma (\bar{x}^*_i - z) + \frac{1}{2} |\bar{x}^*_i - y|^2
 \quad \mbox{subject to \eqref{min on edge}},
\]
where ${\displaystyle \bar{x}^*_i = y - \xi_i}$.


\subsection{Algorithm and discretization}\label{algorithm}

We begin by 
outlining the key steps in Algorithm \ref{algo: w/o dist}. 
First, we set up the initial datum $u_0 \in C(\overline{W})$
and $\theta$ for given initial curve $\Sigma_0$ and its orientation
of the evolution.
Accordingly, the function $u_{n+1}$ for $n \ge 0$ is determined by
\begin{equation}
    u_{n+1} := \arg \min_{w}{E} (w;u_n).
\end{equation}
Then, we obtain the sequence of functions
$u_n$ and then the curves $\Sigma_n$ approximating a solution
$\Sigma(t_n)$ at $t_n = nh$ by 
\begin{equation}
 \label{discrete lv form}
 \Sigma_n 
 := \{ x \in \overline{W} ; \ u_n (x) - \theta (x) \equiv 0
 \mod 2 \pi \mathbb{Z} \}. 
\end{equation}


 








\begin{algorithm}[htbp]
 \caption{Minimizing movement of spiral curves}
 \label{algo: w/o dist}
 \setstretch{1.2}
 \SetKwRepeat{Do}{do}{while}

 \KwIn{$\Sigma_0 \subset \overline{W}$
 and $u_0 \in C(\overline{W})$ such that
 $\Sigma_0 = \{ u_0 - \theta \equiv 0 \}$.}

 \KwOut{$\Sigma (T) = \{ u(T) - \theta \equiv 0 \}$
 with some function $u(T) \colon \overline{W} \to \mathbb{R}$.}

 (Time step) \For{$n=0, 1, \ldots, [T/h]-1$}{
   Set $g = u_n$\;
   Initialize $w^0 = u_n$, $b^0 = d^{0} = 0$\;

   (Outer loop) \For{$k=0,1,\ldots$}{
     Initialize $w^k_0 = w^k$, $d^k_0 = d^k$\;

     (Inner loop) \For{$\ell=0,1,2,\ldots$}{
       Solve \eqref{inner loop eq1}--\eqref{inner loop bc1}
       with $\psi = \gamma (\nabla (u_n - \theta))$ to obtain $w^k_{\ell+1}$\;
 
       Calculate $d^k_{\ell+1}$ with Scheme (Sh)\;
       \lIf{$|F^k_{\mu,\alpha} (w^k_{\ell+1} ,d^k_{\ell+1}; g) 
       - F^k_{\mu,\alpha} (w^k_\ell, d^k_\ell; g)| < \varepsilon_{in}$}{break}
     }
     Set $w^{k+1} = w^k_{j+1}$, $d^{k+1} = d^k_{j+1}$\;
     Set $b^{k+1} = b^k + \nabla w^{k+1} - d^{k+1}$\;
     \lIf{$|F^k_{\mu,\alpha} (w^{k+1}, d^{k+1}; g) 
     - F^k_{\mu,\alpha} (w^k, d^k; g)| < \varepsilon_{out}$}{break}
   }
   Set $u_{n+1} = w^{k+1}$\;
 }
\end{algorithm}

In the algorithm, one can find the following two loops:
\begin{description}
 \item[Outer loop] Finding $(w^{k+1}, d^{k+1})$ from $(w^k,d^k)$ 
	      by \eqref{Bregman iteration: remain form1}--\eqref{Bregman iteration: remain form3},
 \item[Inner loop] Finding $(w^k_{\ell + 1},d^k_{\ell+1})$ 
	      from $(w^k_\ell,d^k_\ell)$
	      by \eqref{split1}--\eqref{split3}.
\end{description}
In these loops, we 
compute the functional 
\begin{align*}
 F_\mu^k (w,d;g)
 := & F(w,d) + \frac{\mu}{2} \| d - \nabla w - b^k \|_{L^2}^2 \\
 = & \int_W \gamma (d - \nabla \theta) dx - \int_W fw dx
 + \frac{1}{2h} \left\| \frac{w - g}{\sqrt{\psi}} \right\|_{L^2}^2 \\
 & \qquad + \frac{\mu}{2} \| d - \nabla w - b^k \|_{L^2}^2
\end{align*}
to check 
if the functions $(w^k_{\ell}, d^k_\ell)$ 
or $(w^k,d^k)$ reaches to the desired minimizer
in the Inner or Outer loop, respectively.
Here,  $\psi = \gamma (\nabla (u_n - \theta))$
 which may be zero in the computation.
To avoid the division by zero in Algorithm \ref{algo: w/o dist}, 
we introduce an approximation of $F_\mu^k$ of the form
\begin{align*}
  F^k_{\mu,\alpha} (w,d;g) 
 := & \int_W \gamma (d - \nabla \theta) dx
 - \int_W fw dx 
 + \frac{1}{2h} \left\| \frac{w - g}{ \sqrt{ \max \{ \psi, \alpha \}}} \right\|_{L^2}^2 
 \\
 & \qquad 
 + \frac{\mu}{2} \| d - \nabla w - b^k \|_{L^2}^2
\end{align*}
with a positive constant $\alpha \ll 1$.
In summary,
we check if
\begin{equation*}
 \left\{
 \begin{aligned}
 \mbox{(for Inner)} \quad & | F^k_{\mu,\alpha} (w^k_{\ell+1},d^k_{\ell+1}; u_n) 
 - F^k_{\mu,\alpha} (w^k_\ell, d^k_\ell; u_n)|
 < \varepsilon_{in}, \\
 \mbox{(for Outer)} \quad & | F^k_{\mu,\alpha} (w^{k+1},d^{k+1}; u_n) 
 - F^k_{\mu,\alpha} (w^k, d^k; u_n)|
 < \varepsilon_{out}  
 \end{aligned}
 \right.
\end{equation*}
for some $\varepsilon_{in}, \varepsilon_{out} \ll 1$
to break the Inner and Outer loop.

On the discretization, we apply the finite difference method.
Note that $\nabla \theta$ should be considered
as a smooth branch of it around the point $x_{i,j}$
at where we discretize the equations or functionals.
See \cite{OTG:2015JSC} for the details to avoid the
discontinuity of $\theta$ in the computation.
The equation \eqref{inner loop eq1}--\eqref{inner loop bc1} can be
solved easily, e.g., by the SOR method.

\subsection{Discretization of the partial derivatives on a grid}
\label{sec: discretization}

We now summarize the discretization 
of $\psi = \gamma(\nabla (u_n - \theta))$
in \eqref{inner loop eq1} for Algorithm~\ref{algo: w/o dist}.
In fact, one can find that
\eqref{inner loop eq1} implies 
$w = g$ when $\gamma (\nabla (g - \theta)) = 0$.
This fact may cause unexpected motions.
For example, 
a small closed curve does not vanish even though 
it should be shrunk because of the curvature.
Such an unexpected motion often appears
when we approximate $\nabla (g - \theta)$ by
the usual central difference.
To avoid such a situation, we introduce an
upwind difference for $\gamma (\nabla (g - \theta))$
as follows.

For the simplicity of notations, we 
set $G_{i,j} = g(x_{i,j}) - \theta (x_{i,j})$.
We present two upwind difference formulas;
the first one is of the form
\begin{align*}
 \hat{\partial}_x^\pm G_{i,j} 
 & = \pm \left\{ \frac{G_{i \pm 1, j} - G_{i,j}}{\Delta x} \right. \\
 & \qquad \left. - \frac{\Delta x}{2} \sigma 
 \left(
 \frac{G_{i \pm 2, j} - 2 G_{i \pm 1, j}
 + G_{i,j}}{\Delta x^2}, 
 \frac{G_{i +1, j} - 2 G_{i, j}
 + G_{i-1,j}}{\Delta x^2}
 \right) \right\}, 
\end{align*}
where
\[
 \sigma (p,q) =
 \left\{
 \begin{array}{ll}
  p & \mbox{if} \ |p| < q, \\
  q & \mbox{otherwise}.
 \end{array}
 \right.
\]
The second one is just a usual difference in the form
\[
 \tilde{\partial}_x^\pm U
 = \pm \frac{G_{i \pm 1, j} - G_{i,j}}{\Delta x}.
\]
Then, we define
\[
 \partial^\pm_x G_{i,j}
 = \left\{
 \begin{array}{ll}
  \hat{\partial}^\pm_x G_{i,j} & \mbox{if} \ x_{i \pm 1, j} \in W
   \ \mbox{and} \ x_{i,j \pm 1} \in W, \\
  \tilde{\partial}^\pm_x G_{i,j} & \mbox{otherwise}.
 \end{array}
 \right.
\]
Note that $G_{i \pm 1,j}$ in $\tilde{\partial}^\pm_x G_{i,j}$,
(resp. $G_{i \pm 2, j}$ in $\hat{\partial}^\pm_x G_{i,j}$)
is calculated by extension of $G_{i,j}$ with the Neumann boundary condition
when $x_{i \pm 1,j} \notin W$ 
(resp. $x_{i \pm 2, j} \notin W$).
We also define $\partial^\pm_y G_{i,j}$
with the similar manner of the above.
Note that we use $\tilde{\partial}_x^\pm G_{i,j}$
in this paper
even when $x_{i, j \pm 1} \notin W$.
We use the same order of the
approximation of $\partial_x (g - \theta)$ and $\partial_y (g - \theta)$.

We now introduce the approximation formula of
$\psi = \gamma (\nabla (g - \theta))$.
For given $q = (q^1, q^2) \in \mathbb{R}^2 \setminus \{ 0 \}$,
let us set
\begin{equation}
 \label{discretization of eikonal term}
 (q \cdot \nabla (g - \theta))_{i,j}
 = 
 \left\{
 \begin{array}{ll}
  q^1 \partial_x^+ G_{i,j} 
   + q^2 \partial_y^+ G_{i,j} 
   & \mbox{if} \ q^1 \ge 0, \ q^2 \ge 0, \\ [4pt]
  q^1 \partial_x^- G_{i,j} 
   + q^2 \partial_y^+ G_{i,j} 
   & \mbox{if} \ q^1 < 0, \ q^2 \ge 0, \\ [4pt]
  q^1 \partial_x^+ G_{i,j} 
   + q^2 \partial_y^- G_{i,j} 
   & \mbox{if} \ q^1 \ge 0, \ q^2 < 0, \\ [4pt]
  q^1 \partial_x^- G_{i,j} 
   + q^2 \partial_y^- G_{i,j} 
   & \mbox{if} \ q^1 < 0, \ q^2 < 0.
 \end{array}
 \right.
\end{equation}
Then, for $\gamma (p) = \max_{0 \le k \le N_\gamma - 1} n_k \cdot p$,
we introduce the approximation 
$\psi_{i,j} \approx \gamma (\nabla (g - \theta) (x_{i,j}))$
in \eqref{inner loop eq1} for Algorithm~\ref{algo: w/o dist}
as
\[
 \psi_{i,j}
 = \max_{0 \le k \le N_\gamma - 1} (n_k \cdot \nabla (g - \theta))_{i,j}.
\]
Hence, we obtain the difference equation for \eqref{inner loop eq1},
which is of the form
\begin{align*}
 w_{i,j} 
 - h \mu \psi_{i,j} 
 \Delta w_{i,j}
 = g_{i,j} + h \psi_{i,j} (f_{i,j} 
 - \mu [\mathrm{div} (d^k_\ell - b^k)]_{i,j}),  
\end{align*}
where $\Delta w_{i,j}$ is the usual finite difference
approximation of $\Delta w$ at $x_{i,j}$.


\section{Numerical results}
\label{Numerics}

We now apply the algorithm introduced
in the previous sections to the evolution of
spirals by \eqref{geo mcf},
and propose some numerical results.

Throughout this section, we set the domain
$\Omega = [-1.5, 1.5]^2$ with the Cartesian grid
\[
 D_\triangle = \{ x_{i,j} = (i \Delta x, j \Delta x) \in \Omega ; \ 
 -1.5/\Delta x \le i,j \le 1.5/\Delta x \}
\]
for a uniform grid spacing $\Delta x$.
We 
will set $\Delta x = 0.02/s$ with $s=1,2,3,4,5$
to examine the numerical accuracy of our algorithms
in \S \ref{sec: unit spiral}, and
often visualize the spiral profiles for $s=3$.

We will further assume that the centers of the spirals lie on the grid nodes: 
$a_\ell \in D_\triangle$ for every $\ell=1, 2, \ldots, N$.
As in the definition of $W$, we will exclude the grid nodes
$x_{i,j}$ 
\[
 |x_{i,j} - a_\ell| < r = {2 \Delta x} 
\]
to avoid the complexity of computations around the centers 
$a_\ell$, $\ell=1, 2, \ldots, N$.

The evolution equation \eqref{geo mcf} in this section 
is rephrased as
\begin{equation}
 \label{BCF eq}
  V_\gamma = v_\infty (1 - \rho_c \kappa_\gamma)
\end{equation}
for consistency with \cite{OTG:2015JSC, OTG:2018CGD} 
or \cite{Burton:1951tr},
where $v_\infty$ and $\rho_c$ are positive constants.
Here, we consider only a constant driving force.
By considering $\tilde{f} = v_\infty$ and 
$\tilde{\gamma} = v_\infty \rho_c \gamma$
instead of $f$ and $\gamma$ in \eqref{geo mcf},
we can apply the algorithms as in the previous sections
to \eqref{BCF eq}.

\subsection{Unit spiral}
\label{sec: unit spiral}

In this section
we consider the fundamental case, i.e., 
a single center with a single spiral,
i.e., 
let $N=1$, $a_1 = 0$, $m_1 = 1$, and thus
we set $\theta (x) = \arg x$.
We choose the size of the center as $r = 2 \Delta x$.
We present numerical studies for the case corresponding to a triangular spiral; 
this means $N_\gamma = 3$ and
\begin{equation}
 \label{triangle setting}
 n_j = 2 \left(
	  \cos \frac{\pi (2j+1)}{3},
	  \sin \frac{\pi (2j+1)}{3} \right),    
 \quad \mbox{for} \ j=0,1,2.
\end{equation}
Note that, in this case, $\gamma^\circ (p) = \max_{0 \le j \le 2} {\tilde{n}_j} \cdot p$
with 
\[
 \tilde{n}_j = \left( \cos (2 \pi j / 3), \sin (2 \pi j / 3) \right).
\]
(See \cite{IO:DCDS-S} for details of
calculation $\gamma^\circ$ from $\gamma$.)
%
We examine the proposed algorithm
for 
\[
 V_\gamma = 1 - 0.01 \kappa_\gamma
 \quad (v_\infty = 1, \ \rho_c = 0.01),
\] 
the time interval is $[0,0.8]$ with 
time span $h = 0.04 \Delta x$.
The other parameters for the presented computational results are
$\mu$, $\varepsilon_{in}$,
$\varepsilon_{out}$, $\alpha$ as follows:
\[
 \mu = v_\infty \rho_c, \quad 
 \varepsilon_{in} = 10^{-2} v_\infty \rho_c, \quad
 \varepsilon_{out} = 10^{-5} v_\infty \rho_c, \quad
 \alpha = 10^{-8}.
\]

We compare the numerical results computed by our algorithm
with those from the front-tracking model introduced in~\cite{IO:DCDS-B}.
A brief review of the front-tracking model is provided in the Appendix.

We will use two functions to quantify the difference 
between the spirals computed by our algorithm 
and those from the front-tracking approach.
We denote the spiral computed by 
Algorithm~\ref{algo: w/o dist} as 
$\Sigma_h (t)$ and the one computed by the front-tracking method 
as $\Sigma_\textup{d} (t)$. 
%
%

The first one, $\mathcal{D} (t)$ is based on the distance function to $\Sigma_h$.
Let $d_{\Sigma_h} (t,x)$ be a distance function 
from $\Sigma_h (t)$ with the usual Euclidean metric.
Then, 
\begin{equation}
 \label{def:D}
 \mathcal{D} (t) 
 = \sup \{ d_{\Sigma_h} (t,x) ; \ x \in \Sigma_{\textup{d}} (t) \}    
\end{equation}
expresses the 
distance between $\Sigma_h (t)$
and $\Sigma_\textup{d} (t)$.

The second one,
denoted by $\mathcal{A} (t)$ (and defined in \eqref{def:A} below), 
measures the area of interposed region by $\Sigma_{\textup{d}} (t)$ and
$\Sigma_h (t)$. 
This concept is introduced in \cite{IO:DCDS-S}.
The function $\mathcal{A}$ involves
the height function of the crystal surfaces.
Let $H_{\Sigma_h}(t,x)$ and 
$H_{\Sigma_\textup{d}} (t,x)$ denote the height functions 
corresponding to $\Sigma_h(t)$ and $\Sigma_{\textup{d}} (t)$ 
(as defined in \cite{IO:DCDS-S}).

For $\Sigma_h (t)$ given by
\eqref{discrete lv form}, we prepare the principal value 
$\hat{\theta} (x) = \arg x \in [-\pi, \pi)$, and let $\zeta (t,x) \in \mathbb{Z}$ 
be such that 
\[
 2 \pi \zeta (t,x) < u_n (x) - \hat{\theta} (x) \le 2 \pi (\zeta (t,x) + 1)
\]
for $(t,x) \in [0,\infty) \times \overline{W}$.
Then, 
\[
 \theta_h (t,x) 
 = \hat{\theta} (x) + 2 \pi \zeta(t,x)
\]
yields the branch of $\arg x$ whose discontinuities are only on $\Sigma_h (t)$.

For $H_{\Sigma_\textup{d}} (t)$, let
\[
 {\bf N}_j = \frac{\tilde{n}_j}{|\tilde{n}_j|}, \quad
 {\bf T}_j = 
 \begin{pmatrix}
  1 & 0 \\
  0 & -1
 \end{pmatrix}
 {\bf N}_j
 \quad \mbox{for} \ j \in \mathbb{Z} / (N_\gamma \mathbb{Z})
\]
denotes the outer unit normal and
tangential vector of $j$-th facet of $\mathcal{W}_\gamma$
with the extended facet number $j \in \mathbb{Z} / (N_\gamma \mathbb{Z})$.
We now consider a polygonal curve
$\Sigma_{\textup{d}} (t) = \bigcup_{j=0}^k \Sigma_{\textup{d},j} (t)$ 
given by
\[
 \Sigma_{\textup{d},j} (t)
 = \left\{
 \begin{aligned}
  \ & \{ (1-\sigma) y_j (t) + \sigma y_{j-1} (t) ; \ \sigma \in [0,1] \}
  \quad \mbox{if} \ j=1,2,\ldots, k, \\
  & \{ y_0 (t) + \sigma {\bf T}_0 ; \ \sigma \ge 0 \}
  \quad \mbox{if} \ j=0
 \end{aligned}
 \right.
\] 
with its vertices $\{ y_j (t) ; \ j = 0, \ldots, k \}$.
Note that
\[
 \frac{y_{j-1} (t) - y_j (t)}{|y_{j-1} (t) - y_j (t)|} =  {\bf T}_j
\]
and then all facets of $\Sigma_{\textup{d}} (t)$ are parallel to those
of $\mathcal{W}_\gamma$.
(See \eqref{discrete facets}--\eqref{discrete origin} and Appendix 
for the detail how to determine $y_j (t)$.)
Let $\tilde{\vartheta}^* = \tilde{\vartheta}_k - \pi/2$, i.e.,
$(\cos \tilde{\vartheta}^*, \sin \tilde{\vartheta}^*) = {\bf T}_k$.
We first prepare a branch of $\arg x$ denoted by 
$\Theta_{\textup{d}} (t,x) \in (\tilde{\vartheta}^* - 2 \pi, \tilde{\vartheta}^*]$
whose discontinuity is only on 
a half line including $\Sigma_{\textup{d}, k} (t)$.
Let 
\[
 D_j (t) = \{ x \in \mathbb{R}^2 ; \ (x-y_j) \cdot {\bf N}_j > 0 \}
\]
for $j=0, \ldots, k$, which denotes the region including the terrace
in front of $\Sigma_{\textup{d},j} (t)$.
Then, 
\[
 \theta_{\textup{d}} (t,x) 
 = \Theta_{\textup{d}} (t,x) - 2 \pi \sum_{j=1}^k \chi (x; D_j^c (t) \cap D_{j-1} (t))
\]
yields the branch of $\arg x$ whose discontinuities are only on $\Sigma_{\textup{d}} (t)$,
where 
\[
 \chi(x;U) = \left\{
 \begin{array}{ll}
  1 & \mbox{if} \ x \in U, \\ 
  0 & \mbox{otherwise} \\
 \end{array}
 \right.
\]
for $U \subset \mathbb{R}^2$. Hence, 
\[
 H_{\Sigma_\textup{d}} (t,x) = \frac{1}{2 \pi} \theta_{\textup{d}} (t,x), 
 \quad 
 {H_{\Sigma_h} (t,x) = \frac{1}{2 \pi} \theta_h (t,x)}
\]
yield 
the surface height functions of $\Sigma_{\textup{d}} (t)$
(resp. $\Sigma_h (t)$) which have 1-jump discontinuity 
on $\Sigma_{\textup{d}} (t)$ (resp. $\Sigma_h (t)$).
Figure~\ref{ode-step sample} is an example of $H_{\Sigma_d}$.
\begin{figure}[htbp]
 \begin{center}
  \includegraphics[scale=0.4]{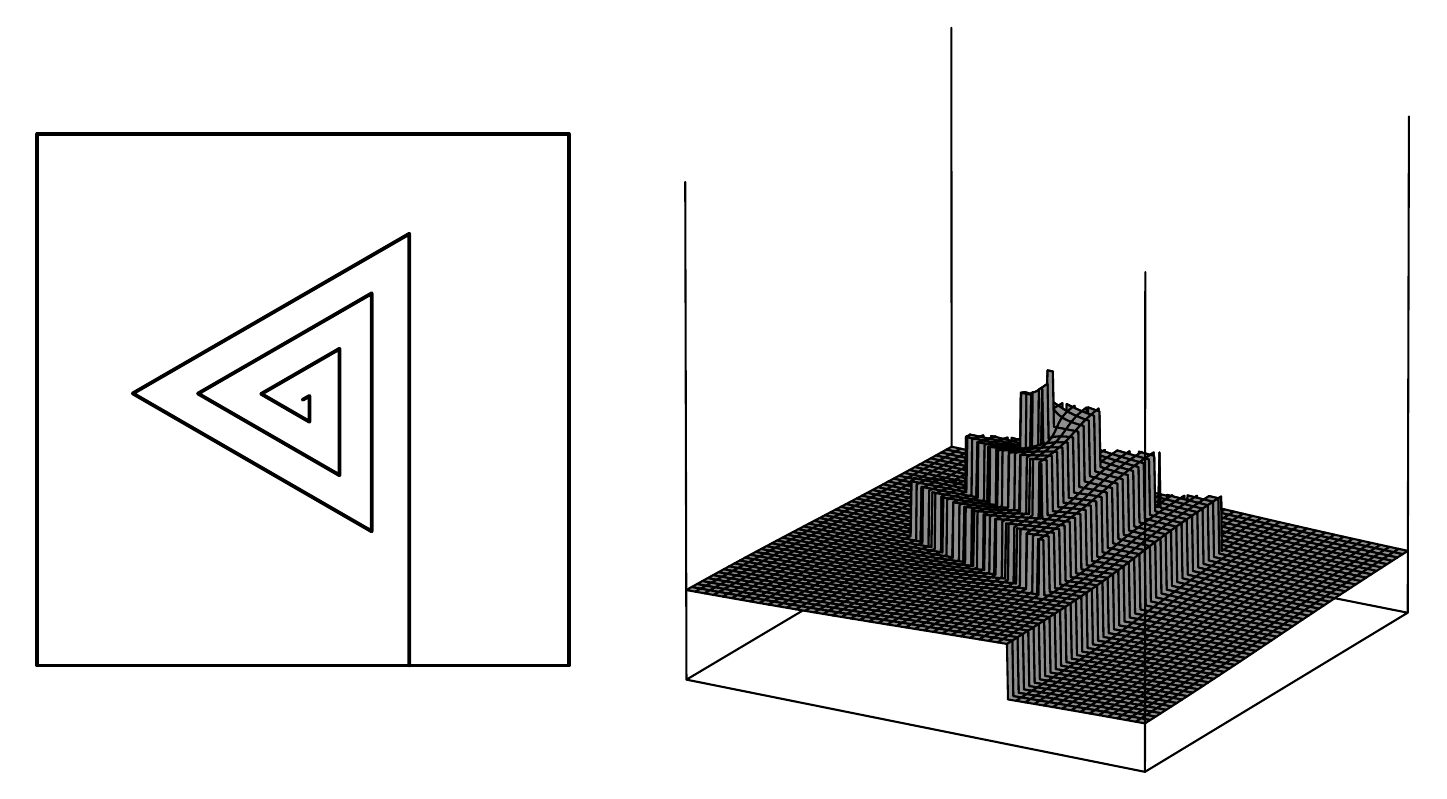}
  \caption{Spiral step function $H_{\Sigma_{\textup{d}}}$.}
  \label{ode-step sample}
 \end{center}
\end{figure}

%
%
%
%
%
%
Now, set up the initial data 
so that $H_{\Sigma_d} = H_{\Sigma_h} $ at $t=0$.
Then, the function $|H_{\Sigma\textup{d}} (t,x) - H_{\Sigma_h} (t,x)|$
plays the role of a characteristic function of the interposed region
between $\Sigma_{\textup{d}} (t)$ and $\Sigma_h (t)$
with multiplicity.
(See Figure~\ref{ode-lv diff sample} for an example
of the graph of $|H_{\Sigma_{\textup{d}}} (t,x) - H_{\Sigma_h} (t,x)|$.)
Hence, we define 
\begin{equation}
 \label{def:A}
 {\mathcal{A} (t)}
 = \frac{1}{|W|} \int_W |H_{\Sigma_\textup{d}} (t,x) - H_{\Sigma_h} (t,x)| dx.
\end{equation}
This quantity indicates the difference between $\Sigma_h (t)$
and $\Sigma_{\textup{d}} (t)$ by area of the interposed region.
\begin{figure}[htbp]
 \begin{center}
  \includegraphics[scale=0.4]{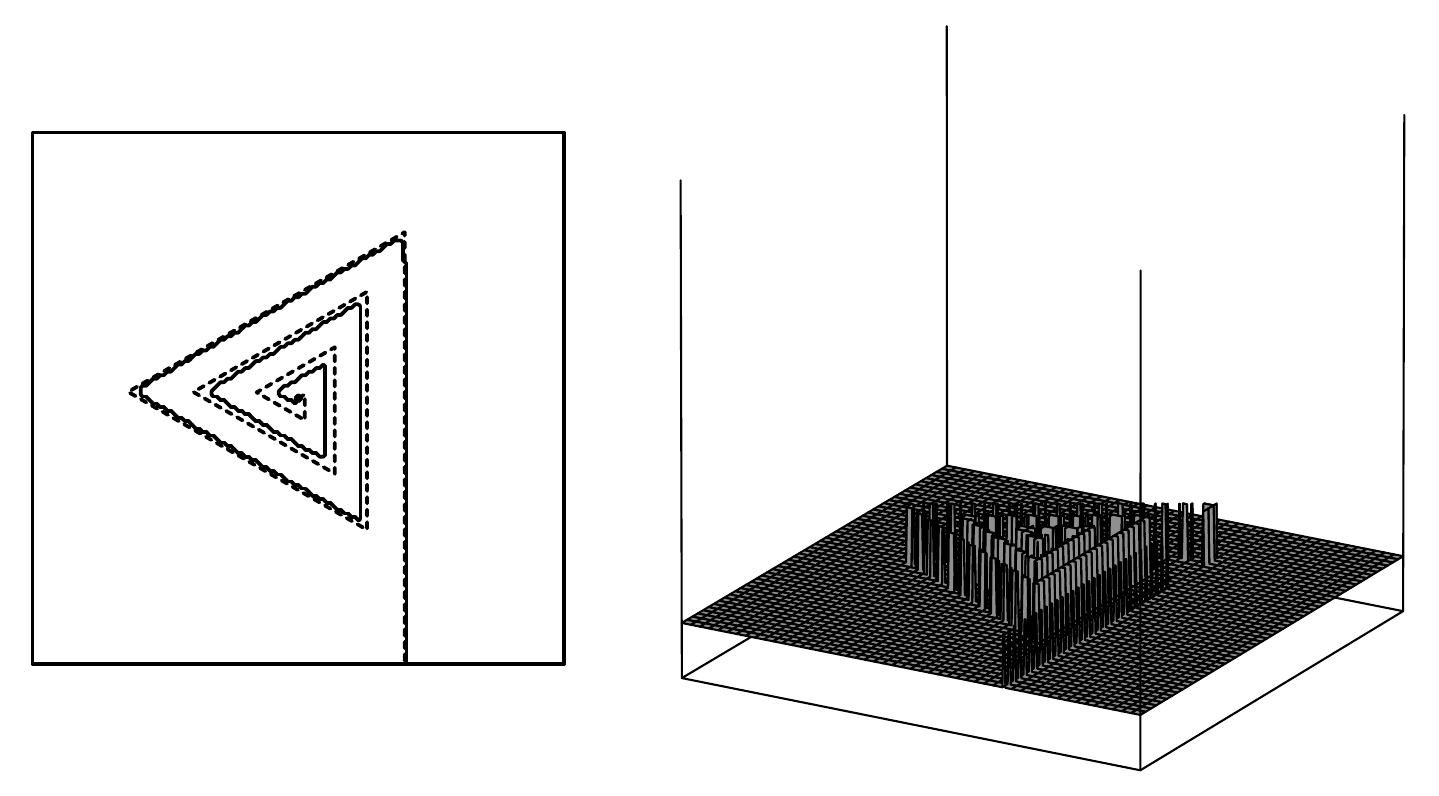}
  \caption{Graph of $|H_{\Sigma_{\textup{d}}} - H_{\Sigma_h}|$.
  In the left panel, the dashed and solid lines indicates
  $\Sigma_{\textup{d}} (t)$ and $\Sigma_h (t)$, respectively.}
  \label{ode-lv diff sample}
 \end{center}
\end{figure}


We take the opportunity to provide some
additional mathematical exposition on $\mathcal{A} (t)$.
Following \cite{Ohtsuka:2003wi}, we introduce a covering space $\mathfrak{X}$
of the form
\begin{align*}
 \mathfrak{X}
 = \left\{ (x, \xi) \in \overline{W} \times \mathbb{R}^N 
 \left| \
 \begin{aligned}
  & \xi = (\xi_1, \ldots, \xi_N) \ \mbox{satisfies} \\
  & (\cos \xi_j, \sin \xi_j) = \mathrm{arg} \frac{x-a_j}{|x-a_j|}
  \ \mbox{for} \ j = 1,2, \ldots, N 
 \end{aligned}
 \right.
 \right\},
\end{align*}
where $a_j, j=1,2,\cdots, N,$ are the centers of the spirals.
We then consider spiral curves
on the Riemannian-like surface
\[
 \mathfrak{W} = \left\{ (x, z) 
 \in \overline{W} \times \mathbb{R} \left| \
 z = \frac{1}{2 \pi} 
 \sum_{j=1}^N m_j \xi_j, \ (x, \xi_1, \ldots, \xi_N) \in \mathfrak{X} \right. \right\}.
\]
In other words, 
the height functions $H_{\Sigma_\textup{d}} (t,x)$ 
of $\Sigma_{\textup{d}} (t)$
divides $\mathfrak{W}$ into the inside
of $\Sigma_{\textup{d}} (t)$ of the form
\[
 \widetilde{I}_{\textup{d}} (t) 
 = \{ (x, \xi) \in \mathfrak{W} | \ \xi \le {H_{\Sigma_{\textup{d}}} (t,x)} \}, 
\]
and the outside as 
$\mathfrak{W} \setminus \widetilde{I}_{\textup{d}} (t)$.
We here define $\widetilde{I}_{h} (t)$,
the inside of $\Sigma_h (t)$
in the same manner with 
$\Sigma_h (t)$ and $H_{\Sigma_h} (t,x)$.
Then, \emph{the interposed region between $\Sigma_{\textup{d}} (t)$
and
$\Sigma_h (t)$} indicates 
$\widetilde{I}_{h} (t)
\triangle \widetilde{I}_{\textup{d}} (t)
= (\widetilde{I}_{h} (t) 
\setminus \widetilde{I}_{\textup{d}} (t))
\cup (\widetilde{I}_{\textup{d}} (t) 
\setminus \widetilde{I}_{h} (t))$.
The quantity $\mathcal{A} (t)$ computes
the 
area of $\widetilde{I}_{h} (t) 
\triangle \widetilde{I}_{\textup{d}} (t)$
in the sense of the usual Lebesgue measure in $\mathbb{R}^2$.
Moreover, note that 
$\mathcal{A} (t)$ includes the multiplicity of the interposed region.
In fact, 
let us consider the situation
$W = \{ x \in \mathbb{R}^2 | \ \rho \le |x| \le R \}$
for some $R > \rho > 0$, $\theta (x) = \arg x$,
$\Sigma_{\textup{d}} (t) = \Sigma_h (t)
= \{ x=(x_1, 0) \in \overline{W} | \ x_1 > 0 \}$,
but
$\widetilde{I}_{\textup{d}} 
= \{ (x, \xi) \in \mathfrak{W} | \ \xi \le 0 \}$,
$\widetilde{I}_{h} 
= \{ (x, \xi) \in \mathfrak{W} | \ \xi \le 2 \}$.
\emph{In this case $\mathcal{D} (t) = 0$.
At first glance, it might seem that the area of
$\{ x \in \overline{W} | \ (x,\xi) \in 
\widetilde{I}_{h} (t) 
\triangle \widetilde{I}_{\textup{d}} (t) \}$
is zero, but
$\mathcal{A} (t) = 2$ (twice of $|W|$).}

We provide several numerical results 
verifying the accuracy of our scheme.
In the following examinations, we set up the initial 
curve as
\[
 \Sigma_{\textup{d}} (0)
 = \Sigma_h (0)
 = \{ \sigma {\bf T}_0 ; \ \sigma \in [0,\infty) \}.
\]
Accordingly, we set the initial data as
\begin{align*}
 \mbox{The front-tracking model}: \ & k=1, \ d_1 (0) = 0. \\
 \mbox{Algorithm~\ref{algo: w/o dist}}: 
 \ & 
 u_0 = - \frac{\pi}{2}.
\end{align*}
In this paper,
the solution $\Sigma_{\textup{d}} (t)$ by the front-tracking model
is obtained by solving the ODE system
\eqref{crystalline ODE} with 4-th order Runge--Kutta method
and smaller time span $\Delta t = 10^{-6}$ than that of our approach.
Then, we compute the height function $H_{\Sigma_{\textup{d}}} (t,x)$
by the solution of front-tracking model.
Figure~\ref{profile: triangle} displays 
overlapped snapshots of the evolving spirals
computed using the front-tracking model 
and using Algorithm~\ref{algo: w/o dist}
at $t = 0, 0.4, 0.6$ and $0.8$, for the case \eqref{triangle setting}.
In Figure~\ref{graph: area-diffs} 
we show the graphs of $\mathcal{D} (t)$ 
and $\mathcal{A} (t)$ under the above settings.
One can find that the difference becomes smaller 
when we choose smaller $\Delta x$ (higher $s$).
However, both $\mathcal{D} (t)$ and $\mathcal{A} (t)$
increase with respect to $t$ for all cases.
From Figure~\ref{profile: triangle},
these differences arise and develop around the center
when the facet associated with the center turns around
the direction of their evolution.
On the other hand, 
one also finds that the graphs of 
$\mathcal{D} (t)$, in particular when $s=1$ or $2$,
seem to be concave even when the graphs of 
$\mathcal{A} (t)$ in the same cases are convex.
This seems to be the situation in which
the difference between the two simulations around the center becomes too large so that 
the step from the front-tracking model catches up 
with some other portion of the 
steps by Algorithm~\ref{algo: w/o dist}. 

\begin{figure}[htbp]
 \begin{center} 
  \includegraphics[scale=0.8]{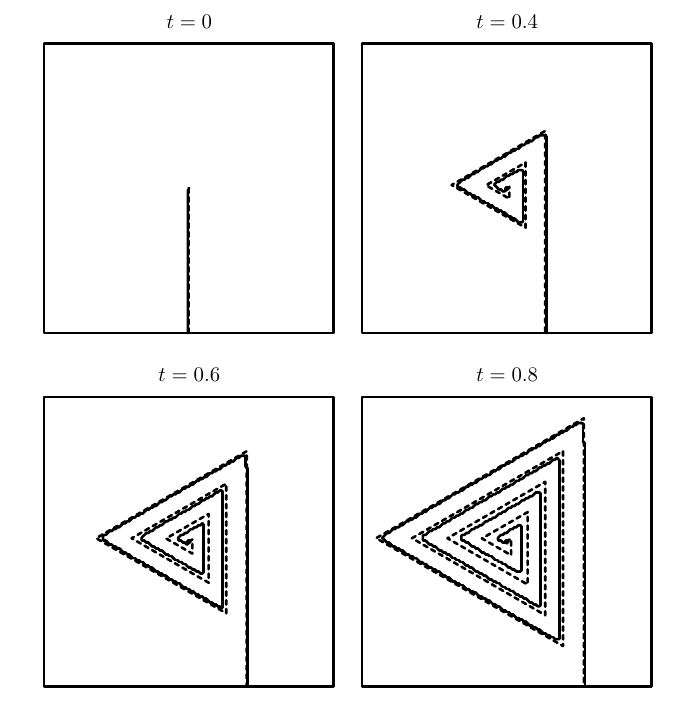}
  \caption{Overlapping snapshots of two spirals 
  at $t = 0, 0.4, 0.6$ and $0.8$, 
  computed using the front-tracking model (dashed curves) 
  and Algorithm~\ref{algo: w/o dist} (solid curves),
  with $\Delta x = 1/150$ ($s=3$). }
  \label{profile: triangle}
 \end{center}
\end{figure}

\begin{figure}[htbp]
 \begin{center}
  \includegraphics[scale=0.6]{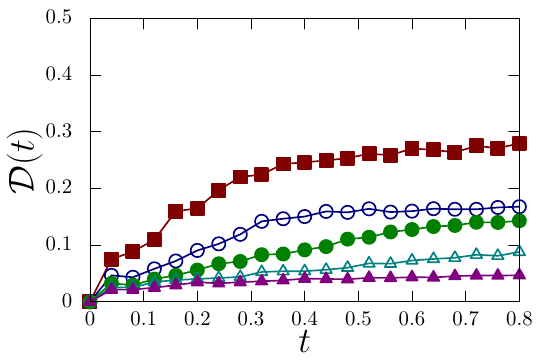} 
  \includegraphics[scale=0.6]{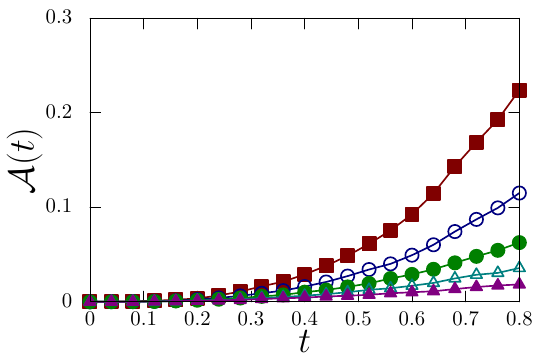} 
  \caption{Graphs of the distances $\mathcal{D} (t)$ and the area difference 
  $\mathcal{A} (t)$ 
  between the spirals computed using the front-tracking model
  and Algorithm~\ref{algo: w/o dist}.
  The lines with $\blacksquare$,
  $\circ$, $\bullet$, $\triangle$, $\blacktriangle$
  correspond the cases of 
  $s=1,2,3,4,5$, respectively. }
  \label{graph: area-diffs}
 \end{center}
\end{figure}

Next, we investigate the number of inner and outer loops 
as the computational cost of our approach.
Figure~\ref{num-outer: triangle}
presents the numbers of outer loops for each time step.
We find that the number of outer loops in a time step ranges from 20 to 80.
On the other hand, Figure~\ref{num-inner: triangle}
present the graphs of average numbers of inner loops par one outer loop
for each time step.
We find that the range of inner loops is from 1 to 1.5 
in all cases.
In particular, the  number of inner and outer loops 
seem to remain very similar for the different values of $\Delta x$ used in the study.



\begin{figure}[htbp]
 \begin{center}
  \includegraphics[scale=0.75]{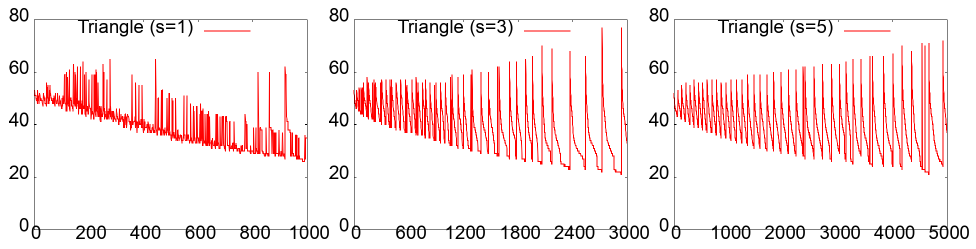}
  \caption{
  The number of outer loops used in the simulations involving three different values of $s$.
  The horizontal axis reveals the number of time steps performed in the evolution, 
  while the vertical axis shows the number of outer loops. 
  }\label{num-outer: triangle}
 \end{center}
\end{figure}



\begin{figure}[htbp]
 \begin{center}
  \includegraphics[scale=0.75]{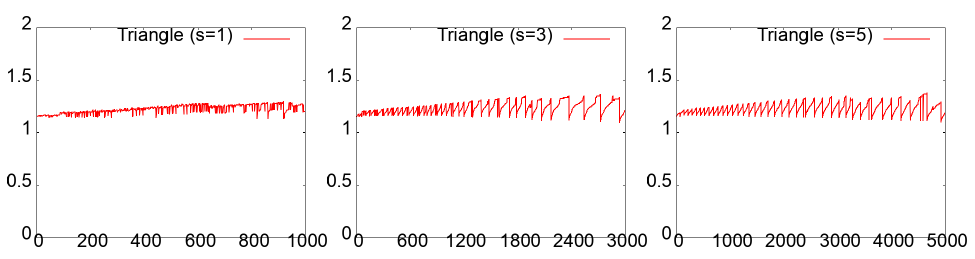}
  \caption{The number of inner loops  used in the simulations involving 
  three different values of $s$.
   The horizontal axis reveals the number of time steps performed in the evolution 
   and vertical axis shows the average number of inner loops per outer loop. 
  }
  \label{num-inner: triangle}
 \end{center}
\end{figure}

We here give some remarks on numerical computation.
For the level set equation \eqref{lv mcf}, the fully explicit finite difference
method is a simple approach to compute it;
see \cite{IO:DCDS-S} for details.
Our numerical results in this paper are close to those in \cite{IO:DCDS-S}.
It is well known that we often set the time span as $h=O(\Delta x^p)$.
For the level set equation for crystalline curvature,
we shall choose $p \ge 2$ for stable and highly accurate 
computation when we choose smaller $\Delta x$.
On the other hand, our approach can be set $h = O(\Delta x)$ ($p=1$),
which is an advantage for computation with smaller $\Delta x$.
On the other hand, our approach has a lot of numerical parameters
affecting to numerical accuracy.
In particular, $\mu$ affects to the profile of spirals and numerical costs.
In Figure~\ref{profile: triangle} one can find that the solution of 
our approach has smoothing effect in particular around the corner
of polygonal curve.
If one wants to obtain sharp corner in the solution,
it is one of options to set smaller $\mu$.
However, if we do so, the number of inner and outer loops may increase. 

We here also give some remarks on the comparison 
between front-tracking model and our approach.
For the computation of the evolution of a single spiral,
the front-tracking model is extremely fast to obtain the solution
$\Sigma_{\textup{d}} (t)$.
If we compute the height function $H_{\Sigma_{\textup{d}}} (t,x)$, 
almost all time of the computation is spent to compute 
$H_{\Sigma_{\textup{d}}} (t,x)$ from $\Sigma_{\textup{d}} (t)$.
Moreover, the front-tracking approach requires generating new facets,
so the ODE system \eqref{crystalline ODE} and thus the facet number $k$ 
becomes larger in time.
Thus, the computational time of the front-tracking model
for one time-step increases in time, 
although that of our approach seems to be constant.

We conclude this subsection by presenting 
a numerical result of the case of a single center
with multiple spirals.
Figure~\ref{sq1-triple w/o profile}
presents the profiles of triple spirals
evolving by
\begin{align}
 \label{eq for triple}
 & V = 2 (1 - 0.04 \kappa_\gamma), \\
 \label{square setting}
 & \gamma (p) = |p^1| + |p^2|, \
 \gamma^\circ (p) = \max \{ |p^1|, |p^2| \}, \ p = (p^1, p^2) \in \mathbb{R}^2.
\end{align}
with $N=1$, $a_1 = 0$, $m_1 = 3$
and thus $\theta = 3 \arg x$.
In our earlier work \cite{Ohtsuka:OWR2017},
an approach using a signed distance 
function in a neighborhood of spirals is investigated.
However, it is difficult to apply that approach in this situation.

%
%
%

\begin{figure}[htbp]
 \begin{center}
  \includegraphics[scale=0.8]{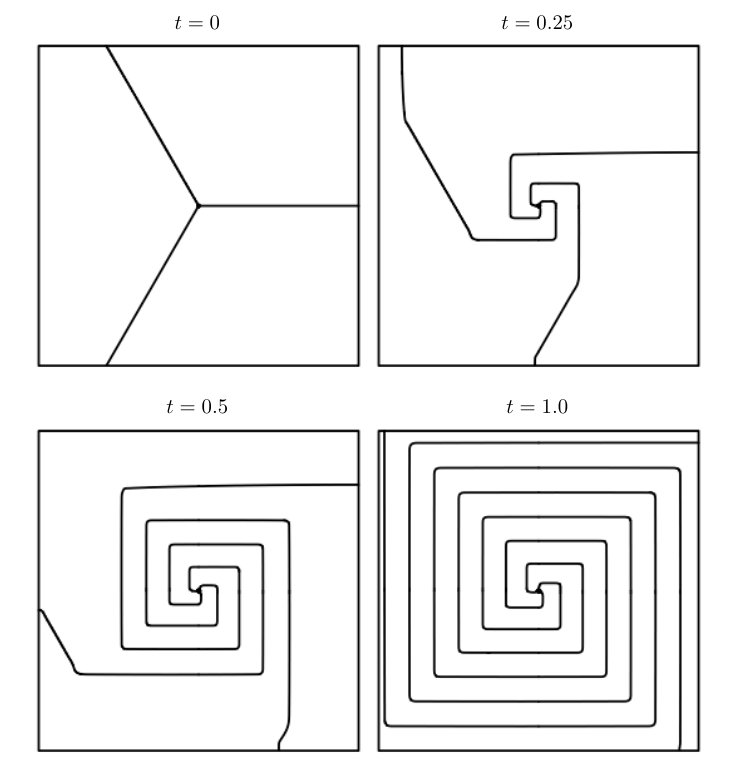}
  \caption{Profiles of triple spirals evolving by \eqref{eq for triple}
  and \eqref{square setting}
  from a single center
  at $t=0$, $t=0.25$, $t=0.5$ and $t=1$.}
  \label{sq1-triple w/o profile}
 \end{center}
\end{figure}

\begin{rem}
 According to the front-tracking model \cite{IO:DCDS-B}, 
 the facet associated with the center (center-facet)
 is stationary if it is shorter than
 the critical length. The critical length in question 
 is the length of the facet of $\mathcal{W}_\gamma$ 
 facing the same direction as the center facet.
 When the length of the center-facet reaches  
 the critical length, a new facet is generated  
 at the center, and newly generated facet becomes the center-facet.
 Our approach does not have such a generation rule.
 However, the behavior of the facet around the center
 by our algorithms is very similar to the generation rule outlined in \cite{IO:DCDS-B}.
 
 Figure~\ref{serial picture for center-facet}
 show a superposition of snapshots of a spiral 
 evolving with
 \begin{equation}
  \label{eq: check generation}
  V = 4(1 - 0.2 \kappa_\gamma),
 \end{equation}
 and \eqref{square setting} by Algorithm~\ref{algo: w/o dist}.
 %
 %
 The dashed line is $x = \pm (0.4 + r)$, $y = \pm (0.4 + r)$.
 The initial curve is $\{ (x,0) ; \ x > 0 \}$.

 We now consider the same situation with the front-tracking model.
 Then, the initial data is
 $\Sigma_{\textup{d}} (0) = \Sigma_{1,\textup{d}} (0) \cup \{ (x,0) ; \ x \ge 0 \}$,
 $\Sigma_{1, \textup{d}} (t) = \{ (0,y) ; \ 0 \le y \le d_1 (t) \}$
 and $d_1 (0) = 0$.
 For the evolution equation \eqref{eq: check generation}
 and \eqref{square setting}, 
 we find $\mathcal{W}_\gamma = [-0.2,0.2]^2$ in this case 
 so that the center-facet $\Sigma_{1,\textup{d}} (t)$ 
 can move after the time $t^* > 0$ when $d_1 (t^*) = 0.4$.
 In other words, 
 $\Sigma_{\textup{d}} (t)$ evolves as the following:
 \begin{enumerate}
  \item First, the horizontal line 
	$\{ (x,0) ; \ x \ge 0 \}$ moves to 
	the above direction 
    as $\{ (x, d_1 (t)) ; \ x \ge 0 \}$
	with making the vertical line
	$\Sigma_{1, \textup{d}} (t) = \{ (0, y) ; \ 0 \le y \le d_1 (t) \}$
	while $d_1 (t) < 0.4$.
  \item When $d_1 (t) = 0.4$, the new facet 
	$\Sigma_{2,\textup{d}} (t) = \{ (x,0) ; \ - d_2(t) \le x \le 0 \}$ 
	is generated.
	The center facet of $\Sigma_{\textup{d}} (t)$ is changed 
	from $\Sigma_{1,d}$ to $\Sigma_{2,d}$, and then
	$\Sigma_{1,d}$ can move as 
	$\Sigma_{1,d} = \{ (-d_2 (t), y); \ 0 \le y \le d_1 (t) \}$.
 \end{enumerate}
 The motion of $\Sigma_h (t)$ is similar to the above scenario.
 The vertical line in Figure~\ref{serial picture for center-facet}
 moves very slowly until the horizontal line reaches the dashed line
 ($y=0.4+r$).
 After the horizontal line goes over the dashed line, 
 the vertical line accelerates.
 This seems to be the reason why we observe that 
 $\Sigma_{\textup{d}} (t)$ and
 $\Sigma_h (t)$ are very close in numerical sense.
\end{rem}

\begin{figure}[htbp]
 \begin{center}
  \includegraphics[scale=0.5]{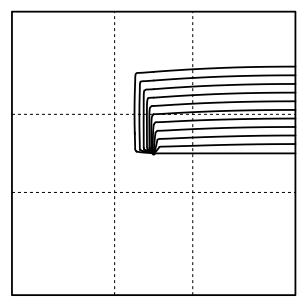}
  \caption{Serial picture of an evolving spiral
  from $t = 0$ to $t = 0.25$ per $\Delta t = 0.025$,
  whose center is the origin,
  by 
  \eqref{eq: check generation} and \eqref{square setting}.
  The dashed lines denote $x = \pm (0.4 + r)$ and 
  $y = \pm (0.4 + r)$.}
  \label{serial picture for center-facet}
 \end{center}
\end{figure}

\subsection{Merging of spirals}

Co-rotating and opposite rotational pairs are given by
setting
\[
 \theta (x) = m_1 \arg (x - a_1) + m_2 \arg (x-a_2);
\]
set $m_1 m_2 > 0$ for a co-rotating pair,
or $m_1 m_2 < 0$ for an opposite rotational pair.
Figure~\ref{profile:sq1-coro_noD} shows the profiles
of co-rotating spirals at $t=0$, $0.25$, 
$0.39$, $0.40$, $0.50$ and $1.0$
evolving by 
\begin{equation}
 \label{sq1-coro eq}
 V=2(1 - 0.1 \kappa_\gamma),
\end{equation}
with $\gamma$ defined in \eqref{square setting}.
The other parameters are as follows.
\begin{itemize}
 \item Centers: $a_1 = (-0.7,0)$, $a_2 = (0.7,0)$,
       $m_1 = m_2 = 1$.
 \item Initial curve: $\Sigma_0 = \{ a_1 + (-x,0) \in \overline{W} ; \ x>0 \}
       \cup \{ a_2 + (x,0) \in \overline{W} ; \ x > 0 \}$,
       and thus $u_0 (x) = 0$.
\end{itemize}
The panels of $t = 0.39$ and $0.40$ in Figure~\ref{profile:sq1-coro_noD}
shows the detailed profiles when the spirals merge.
Our approach can keep very narrow parallel facing steps 
just before merging, at $t = 039$.
It means that our approach has a good performance
to reduce unwanted smoothing effect.
Note that $\nabla (u - \theta) = 0$ at where 
facing steps merge.
Thus, one may find that an unexpected hole
(closed curve whose direction is inside of the curve)
remains in such a situation, 
in particular when we use the usual central difference
for the computation of $\psi = \gamma (\nabla (u_n - \theta))$
in Algorithm~\ref{algo: w/o dist}.
We use a higher-order upwind discretization 
\eqref{discretization of eikonal term} to avoid this 
unwanted effect. 

\begin{figure}[htbp]
 \begin{center}
  \includegraphics[scale=0.8]{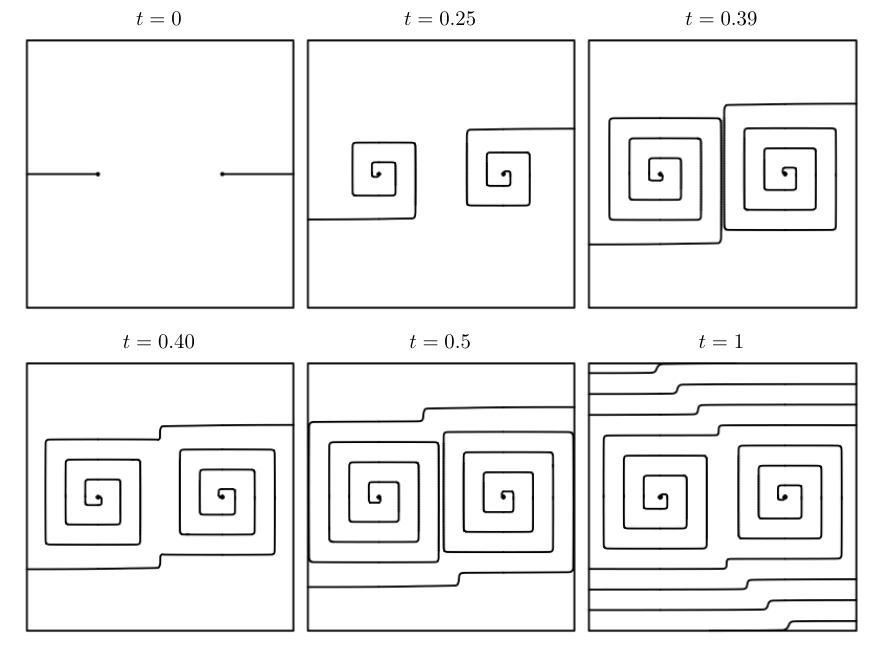}
  \caption{Profiles of co-rotating spirals at $t=0$, $0.25$, $0.5$ 
  and $1.0$ evolving by \eqref{sq1-coro eq} 
  with Algorithm~\ref{algo: w/o dist}.
  Note that $\Delta x = 1/150$($s=3$).}
  \label{profile:sq1-coro_noD}
 \end{center}
\end{figure}



In the works
\cite{AngenentGurtin:1989ARMA, Taylor:1991, Ishiwata:2014, IO:DCDS-B},
the solution is defined in a class of admissible polygonal curves,
which satisfy
\begin{itemize}
 \item there is no self-intersection, 
 \item if a facet of the curve has the same direction as the normal
       ${\bf n} = {\bf N}_j$ as the $j$-th facet
       of $\mathcal{W}_\gamma$,
       its adjacent facet has the direction of a normal
       either ${\bf n} = {\bf N}_{j-1}$ or ${\bf n} = {\bf N}_{j+1}$.
\end{itemize}
The second property of the above is required to define 
the crystalline curvature for $\mathcal{W}_\gamma$.
However, the merging of spirals makes the situation 
breaking the admissibility of the curve;
two adjacent facets which are parallel to the facets
of the Wulff shape $\mathcal{W}_\gamma$, but skipping
the order of the facets of $\mathcal{W}_\gamma$.
Figure~\ref{coro-pent} shows
profiles of co-rotating pentagonal spirals
at $t=0, 0.2, 0.4$, and $0.6$ evolving by
\begin{equation}
 \label{eq: coro-pent}
  \left\{
  \begin{aligned}
   & V_\gamma = 3(1 - 0.01 \kappa_\gamma), \\
   & \mbox{with} \quad 
   \gamma (p) = \max_{0 \le j \le 4} n_j \cdot p, \quad 
   n_j = \left( \cos \frac{\pi(2j+1)}{5}, \sin \frac{\pi (2j+1)}{5} \right).
  \end{aligned}
  \right.
\end{equation}
The two centers are at
$a_1 = (-1,0.4)$ and $a_2 = (1, -0.4)$
with $m_1 = m_2 = 1$.
Note that the above situation provides a regular pentagonal 
$\mathcal{W}_\gamma$ whose facet has the normal 
${\bf N}_j = \left( \cos (2 \pi j/5), \sin (2 \pi j/5) \right)$
for $j=0,1,2,3,4$.
We choose the initial curve as 
\begin{equation}
 \label{init: coro-pent}
 \Sigma_0 = \{ a_1 + \sigma {\bf T}_0 ; \ \sigma > 0 \}
 \cup \{ a_2 + {\sigma} {\bf T}_2 ; \ {\sigma} > 0 \}. 
\end{equation}
In this situation, a vertex of the curve provided from $a_2$ touches to 
the flat portion of the curve from $a_1$ between $t=0.2$ and $t=0.4$.
At that time, the curve is divided into two parts, which are the merging
curves between the facets parallel to ${\bf T}_0$ and ${\bf T}_2$,
or ${\bf T}_0$ and ${\bf T}_3$.
Hence, one can find that merging spirals
may break the admissibility of the curve.
In such cases, the ODE approach due to \cite{IO:DCDS-B} needs
a procedure adding several zero-length facets fulfilling
admissibility;
see \cite{Ochiai:thesis} or \cite{GGH:2006SIMA}.
Such an intermediate facet also seems to appear
in the evolving curve by our methods,
although it is too hard to find in Figure~\ref{coro-pent}.
To clarify the behavior around the merging of spirals,
we examine the situation 
\begin{equation}
 \label{eq: coro-pent2}
 \left\{
  \begin{aligned}
   & V=1 - 0.1 \kappa_\gamma \quad \mbox{with the same anisotropy
   of \eqref{eq: coro-pent}}, \\
   & a_1 = (-0.2, 1), \ a_2 = (0.2, -1), \ m_1 = m_2 = 1,
  \end{aligned}
 \right.
\end{equation}
Figure~\ref{coro-pent2} presents the profiles of 
the evolving curve by \eqref{eq: coro-pent2}
with the same initial curve as \eqref{init: coro-pent}.
One can find the intermediate facet parallel to ${\bf T}_1$
appears in the merging curves between the facets 
facets parallel to ${\bf T}_0$ and ${\bf T}_2$;
see the panel of $t=0.8$ or $t=1$ in Figure~\ref{coro-pent2}.

\begin{figure}[t]
 \begin{center}
  \includegraphics[scale=0.8]{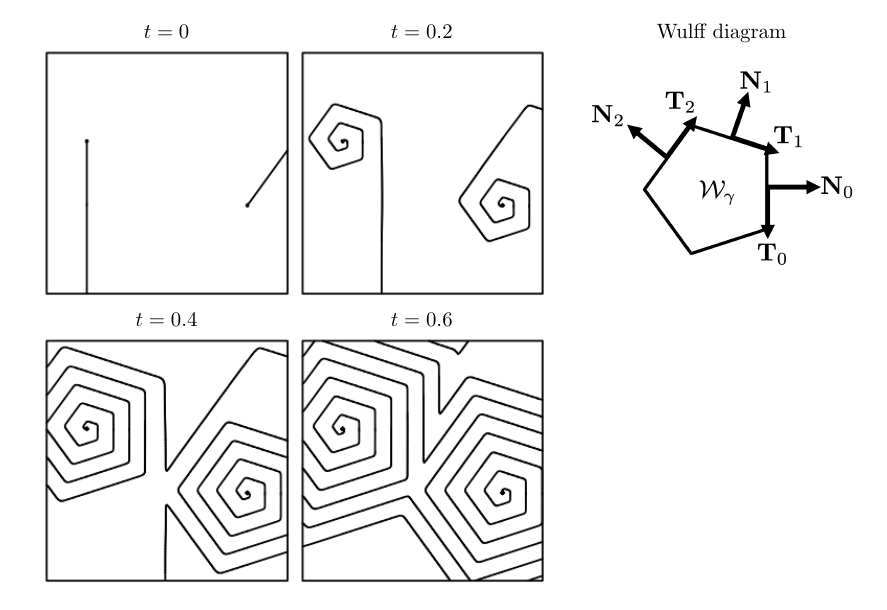}
  \caption{Profiles of co-rotating spirals
  evolving by \eqref{eq: coro-pent}
  at $t=0, 0.2, 0.4, 0.6$ with Algorithm~\ref{algo: w/o dist},
  and the Wulff diagram $\mathcal{W}_\gamma$ of this case.}
  \label{coro-pent}
 \end{center}
\end{figure}

\begin{figure}[t]
 \begin{center}
  \includegraphics[scale=0.8]{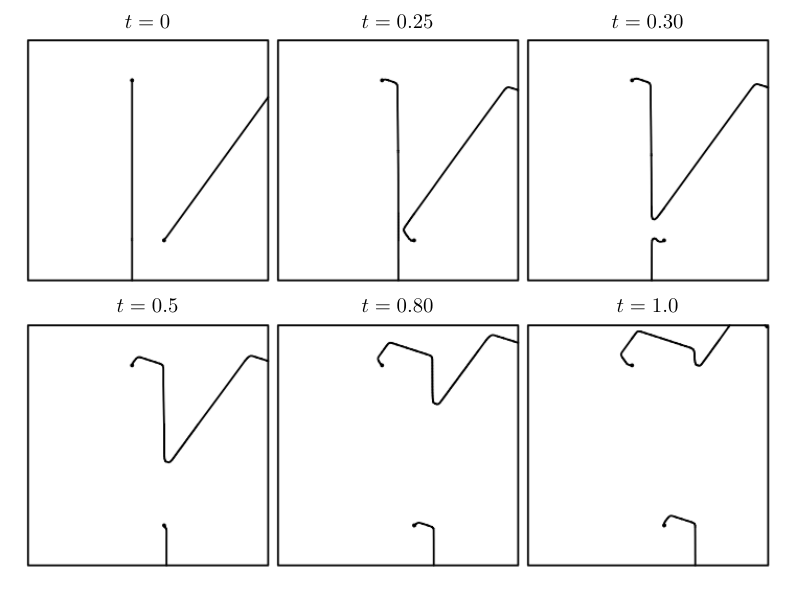}
  \caption{Profiles of merging spirals breaking admissibility.
  The settings of this situation is in \eqref{eq: coro-pent2},
  and $\Delta x = 0.00667$ ($s=3$).}
  \label{coro-pent2}
 \end{center}
\end{figure}

\subsection{Combination of anisotropy}

Let us consider the situation such that
the anisotropies of crystalline curvature and eikonal term,
which are denoted by $\gamma_1$ and $\gamma_2$,
are possibly different.
In other words, we now solve
the level set equation 
\eqref{lv mcf} of the form
\[
 u_t - \gamma_1 (\nabla (u - \theta))
 \left\{ \mathrm{div} [\xi_2 (\nabla (u - \theta))] + f \right\}
 = 0 \quad 
 \mbox{in} \ (0,T) \times W,
\]
where $\xi_2 = \nabla \gamma_2$. 
For such an equation, our approach can easily be adapted as follows;
 \begin{enumerate}
  \item Solve \eqref{inner loop eq1}--\eqref{inner loop bc1}
	with $\psi = \sqrt{\gamma_1 (\nabla (u_n - \theta))}$.
  \item Solve \eqref{polyhedral shrinkage problem}
	by the scheme (Sh) with $\gamma = \gamma_2$.
 \end{enumerate}
Figure~\ref{profile: sq-pent_noD} shows the profiles
of a spiral at $t=0$, $0.25$, $0.50$ and $1.0$ evolving by
\begin{equation}
 \label{eq: sq-pent}
  \left\{
   \begin{aligned}
    & V_1 =2 (1 - 0.01 \kappa_2), \\
    & \gamma_1 (p)
    = \max_{0 \le j \le 4} n_{1,j} \cdot p, \quad 
    n_{1,j} = \left( \cos \frac{\pi (2j+1)}{5}, \sin \frac{\pi (2j+1)}{5} \right), \\
    & \gamma_2 (p) 
    = \max_{0 \le j \le 3} n_{2,j} \cdot p, \quad
    n_{2,j} = \sqrt{2} \left( \cos \frac{\pi (2 j + 1)}{4}, 
    \sin \frac{\pi (2j + 1)}{4} \right)
    \quad (\mbox{case \eqref{square setting}})
   \end{aligned}
  \right. 
\end{equation}
with Algorithm~\ref{algo: w/o dist}, 
where $V_1$ and $\kappa_2$ denotes
the normal velocity defined by $\gamma_1$
and the crystalline curvature defined by $\gamma_2$,
respectively. The Wulff diagrams $\mathcal{W}_1$ by $\gamma_1$
and $\mathcal{W}_2$ by $\gamma_2$ are respectively
regular pentagon and square having facets whose direction normal
are as follows: 
\[
 \mathcal{W}_1 :
 {\bf N}_{1,j} = \left( \cos \frac{2 \pi j}{5}, \sin \frac{2 \pi j}{5} \right),
 \quad 
 \mathcal{W}_2 :
 {\bf N}_{2,j} = \left( \cos \frac{\pi j}{2}, \sin \frac{\pi j}{2} \right).
\]
Thus, the evolving spiral forms an octagonal polygonal curve
whose direction normals are chosen from 
$\{ {\bf N}_{1,j} ; \ 0 \le j \le 4 \}
\cup \{ {\bf N}_{2,j} ; \ 0 \le j \le 3 \} $.
Yazaki \cite{Yazaki:2001}
shows that a curve evolving by ${V} = - \kappa_\gamma + 1$
asymptotically 
converges to the solution of ${V} = 1$,
after the curve has moved 
sufficiently far away.
From the context of this result, the evolving spiral 
will show the profile as follows:
\begin{itemize}
 \item It is close to the rescaling of $\mathcal{W}_2$ around the center.
 \item It becomes convex polygonal curve whose direction normals are chosen
       from ${\bf N}_{1,j}$ or ${\bf N}_{2,j}$
       in the intermediate region.
 \item Asymptotically, it is 
       close to the rescaling of $\mathcal{W}_1$ on far away region.
\end{itemize}
\begin{figure}[htbp]
 \begin{center}
  \includegraphics[scale=1.0]{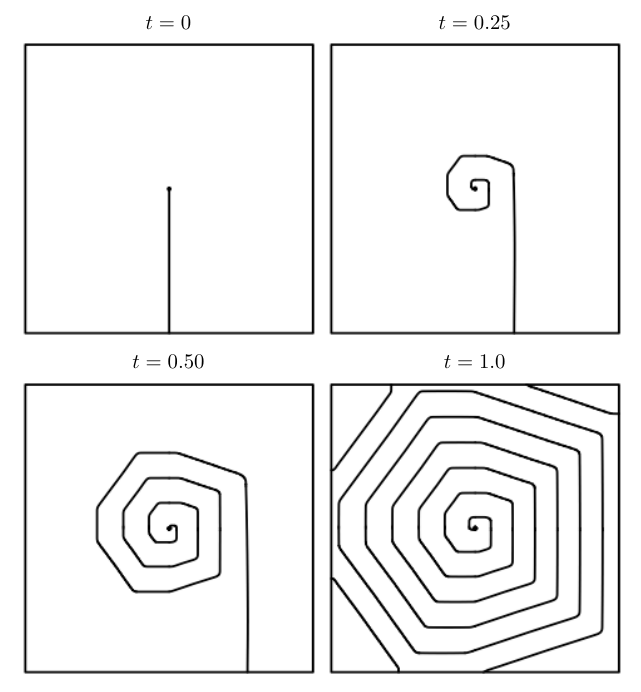}
  \caption{Profiles of a spiral
  at $t=0$, $0.25$, $0.50$ and $1.0$ 
  evolving by \eqref{eq: sq-pent} 
  with Algorithm~\ref{algo: w/o dist}
  and $\Delta x = 1/150$ ($s=3$).}
  \label{profile: sq-pent_noD}
 \end{center}
\end{figure}

\subsection{Interlace motion}
\label{sec: interlace}

Let us consider the situation such that
all centers have the same multiple numbers of spirals,
i.e., $m_j = \tilde{m}_j m_0$ with $\tilde{m}_j \in \mathbb{Z}$
and $m_0 \ge 2$ for $j=1, \ldots, N$.
Then, the level set description of spiral curves
\[
 \Sigma (t) = \{ x \in \overline{W} ; \ 
 u(t,x) - \theta (x) \equiv 0 \mod 2 \pi \mathbb{Z} \}
\]
is divided by $m_0$ components of continuous curves
$\Sigma_\ell (t)$ for $\ell = 0, 1, \ldots, m_0 - 1$
as 
\begin{align}
 \label{layer curve}
 \Sigma (t) & = \bigcup_{\ell = 0}^{m_0 - 1} \Sigma_\ell (t), \\
 \label{layer lv-form}
 \Sigma_\ell (t) & = 
 \{ x \in \overline{W} ; \ 
 u(t,x) - \theta (x) \equiv 2 \pi \ell \mod 2 \pi m_0 \mathbb{Z} \}.
\end{align}
In such case, we attempt to consider the situation
such that each $\Sigma_\ell (t)$ evolves its own evolution equation
\begin{equation}
 \label{layer eqs}
  V_{\gamma_\ell} = f_\ell - \kappa_{\gamma_\ell}
 \quad \mbox{on} \ \Sigma_{\ell} (t)
\end{equation}
by an anisotropic curvature $\kappa_{\gamma_\ell}$
with an anisotropic energy density $\gamma_\ell$
and a driving force $f_\ell$
for $\ell = 0, 1, \ldots, m_0 - 1$.
We call such an evolution interlace motion.
One can find such a situation in some real crystal growth experiments; see, 
for example, \cite{Verma:PhylosMag1951}.
Its growth mechanism is understood as in \cite{vanEnckevort:2004ActaCry}.

We now give a formal discussion to examine the interlace motion
with our algorithm. 
For 
simplicity of exposition,
we tentatively assume that
$\gamma_\ell$ are smooth so that the level set equation
\begin{align*}
 & 
 u_t 
 + F_\ell (\nabla (u - \theta), \nabla^2 (u - \theta))
 = 0
 \\
 & \mbox{where} \quad
 F_\ell (\nabla w, \nabla^2 w) = \gamma_\ell (\nabla w)
 \{ \mathrm{div} ( {\xi_\ell} 
 (\nabla w)) + f_\ell \}, \quad
 {\xi_\ell = \nabla \gamma_\ell}
\end{align*}
for \eqref{layer eqs} is well-defined for each $\ell = 0, 1, \ldots, m_0-1$.
%
%
Due to the idea by \cite{TsaiGiga:2003}, 
we now introduce a periodic cut-off function
$\lambda \colon \mathbb{R} / (2 \pi m_0 \mathbb{Z}) \to \mathbb{R}$;
define 
\[
 \lambda (\sigma) = 
 \max \left\{ 0, \min \left\{ 1, \frac{\pi - |\sigma|}{\varepsilon} + \frac{1}{2} \right\} \right\} 
\]
for a constant $0 < \varepsilon \ll 1$.
Note that $\lambda$ satisfies
\begin{align*}
 & \lambda (\sigma) = \left\{
 \begin{array}{ll}
  1 & \mbox{if} \ \sigma \in [-\pi + \varepsilon/2, \pi - \varepsilon/2], \\
  0 & \mbox{if} \ \sigma \in [-m_0 \pi, m_0 \pi] \setminus 
   (-\pi - \varepsilon/2, \pi + \varepsilon/2),
 \end{array}
 \right. \\
 & \sum_{\ell = 0}^{m_0-1} \lambda (\sigma - 2 \pi \ell) = 1.
\end{align*}
Then, formally, we find that
the motion by \eqref{layer eqs}
will be formulated 
with the solution $u(t,x)$ to  
\begin{align}
 & 
 u_t 
 + \lambda (u - \theta) F_0 (\nabla (u - \theta), \nabla^2 (u - \theta))
 \nonumber \\
 & \quad 
 + \lambda (u - \theta - 2 \pi) F_1 (\nabla (u - \theta), \nabla^2 (u - \theta))
 \nonumber \\
 & \qquad
 + \cdots 
 + \lambda (u - \theta - 2 \pi (m_0 - 1)) 
 F_{m_0 - 1} (\nabla (u - \theta), \nabla^2 (u - \theta))
 = 0
 \label{layer lv eq}
\end{align}
in $(0,T) \times W$.
Now, we consider the implicit difference scheme 
of \eqref{layer lv eq} only on time variable.
Then, we obtain
\begin{align*}
 & u(t+h) \\
 & = u(t) + 
 h 
 \sum_{\ell = 0}^{m_0 - 1}
 \lambda (u (t+h) - \theta - 2 \pi \ell) 
 F_\ell (\nabla (u (t+h) - \theta), \nabla^2 (u (t+h) - \theta)).
\end{align*}
According to \S \ref{sec: w/o dist},
the minimizer $w^*_\ell$ of
\begin{equation*}
 w \mapsto 
 \int_W \gamma_\ell (\nabla (w - \theta)) dx
 - \int_W f_\ell w dx 
 + \frac{1}{2h}
 \left\| \frac{w - u(t)}{\sqrt{\gamma_\ell (\nabla (u(t) - \theta ))}}
 \right\|_{L^2}^2
\end{equation*}
satisfies
\begin{align*}
 w^*_\ell & = u(t) + h \gamma_\ell (\nabla (u(t) - \theta))
 \left\{ 
 \mathrm{div} \{ 
 {\xi_\ell} (\nabla (u(t+h) - \theta) ) \} 
 + f_\ell \right\} \\
 & \approx u(t) + h F_\ell (\nabla (u(t+h) - \theta), \nabla^2 (u(t+h) - \theta)).
\end{align*}
Hence, we observe that
\begin{align*}
 w^* & = \sum_{\ell = 0}^{m_0-1}
 \lambda (u(t) - \theta - 2 \pi \ell) w^*_\ell \\
 & \approx 
 u(t) + h 
 \sum_{\ell=0}^{m_0 - 1}
 \lambda (u (t) - \theta - 2 \pi \ell)
 F_\ell (\nabla (u(t+h) - \theta), \nabla^2 (u(t+h) - \theta)) 
\end{align*}
gives a finite difference scheme of
\eqref{layer lv eq}.
Thus, we define an algorithm of 
the interlace motion by \eqref{layer curve}, \eqref{layer lv-form}
and \eqref{layer eqs} as follows.
\begin{enumerate}
 \item \label{layer min-move}
       For given $\Sigma_n$ ($n \ge 0$) and
       $u_n \in C(\overline{W})$ satisfying
       \begin{align*}
	\Sigma_n & = \{ x \in \overline{W} ; \ 
        u_n (x) - \theta (x) \equiv 0 \mod 2 \pi \mathbb{Z} \} 
	\quad {\bf n} = - \frac{\nabla (u_n - \theta)}{|\nabla (u_n - \theta)|}, \\
	\Sigma_{n, \ell}
	& = \{ x \in \overline{W} ; \ 
        u_n (x) - \theta (x) \equiv 2 \pi \ell \mod 2 \pi m_0 \mathbb{Z}
	\} \quad \mbox{for} \ \ell = 0, \ldots, m_0 - 1,
       \end{align*}
       find the minimizer $w^*_\ell \in L^2 (W) \cup BV(W)$ of
       \begin{align*}
	w \mapsto & \int_W \gamma_\ell (\nabla (w - \theta)) dx
	- \int_W f_\ell w dx 
	+ \frac{1}{2h}
	\left\| \frac{w - u_n}{\sqrt{\max \{ \psi_\ell, \alpha \} }}
	\right\|_{L^2}^2
       \end{align*}
       with $\psi_{\ell} = \gamma_\ell (\nabla (u_n - \theta))$
       for every $\ell = 0, 1, \ldots, m_0 - 1$.
 \item Set 
       \[
	u_{n+1} = \sum_{\ell=0}^{m_0 - 1} \lambda (u_n - \theta
        - 2 \pi \ell) w^*_\ell
       \]
       to return \eqref{layer min-move}.
\end{enumerate}
Figure~\ref{tri-interlace} 
presents an example of interlace motion
by two minor steps with a triangle anisotropy
successively rotating $\pi/2$.
The details of the setting are as follows.
\begin{equation}
 \label{set: e-tri-interlace}
  \left\{ \ 
  \begin{aligned}
   \mbox{Center}: \ & N=1, \ a_1 = O, \ \theta (x) = 2 \arg x, \\
   \mbox{Evolution equation}: \ & V_\ell = 3 (1 - 0.02 \kappa_\ell), \\
   \mbox{Energy density}: \  
   & \gamma_\ell (p) 
   = \max_{0 \le j \le 2} (\cos \vartheta_{\ell, j}, \sin \vartheta_{\ell, j}) 
   \cdot p \\
   & \mbox{with} \
   \vartheta_{\ell, j} = \frac{2 \pi j}{3} + \frac{\pi}{6} + \frac{\pi \ell}{2}
   \ (j=0,1,2, \ \ell = 0,1), \\
   \mbox{Initial data}: \ &
   \mbox{The curve obtained from} \
   u_0 = 0.
  \end{aligned}
  \right.
\end{equation}
One can find a hexagonal pattern made by evolution 
and bunching of curves with triangle anisotropy.
\begin{figure}[htbp]
 \begin{center}
  \includegraphics[scale=1.0]{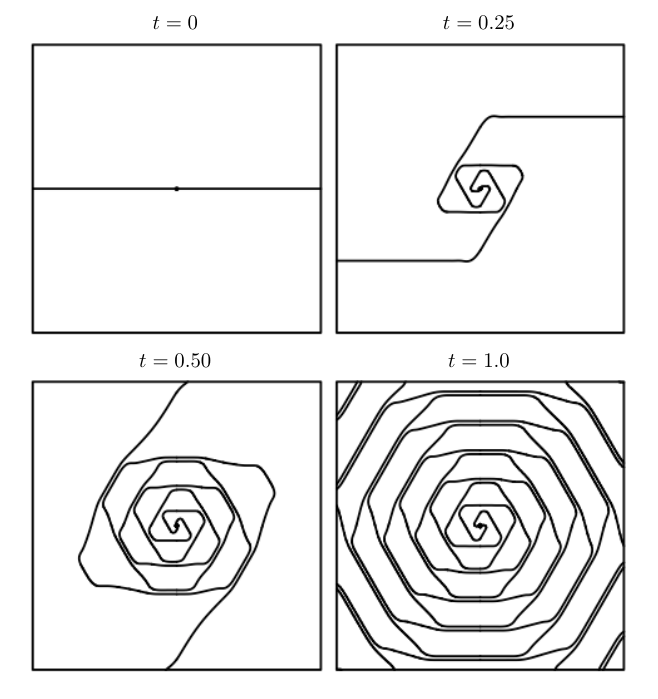}
  \caption{Profiles of interlace motion at 
  $t=0, \ 0.25, \ 0.5$, and $1.0$ by \eqref{set: e-tri-interlace}.}
  \label{tri-interlace}
 \end{center}
\end{figure}

Shtukenberg et al. \cite{Shtukenberg:2013}
reports fascinating phenomena
so-called illusory loops and spirals.
According to \cite{Shtukenberg:2013},
L-cystine crystals have the following features:
\begin{itemize}
 \item The crystal lattice has a hexagonal structure
       ($N_\gamma = 6$).
 \item Unit cell of lattice consists of 6 layers
       ($m_0 = 6$)
       successively rotated clockwise by $\pi/3$. 
 \item Each minor layer has a hexagonal pattern,
       whose unit cell has one facet having
       small energy density, and five facets having
       large energy density.
\end{itemize}
By the above structure, screw dislocation 
with the height of the unit molecule in an L-cystine crystal
provides {6} minor steps evolving by hexagonal crystalline motion. 
Moreover, because of the asymmetry of energy density
and twisted structure of it, a bunch of minor steps forms
\begin{itemize}
 \item hexagonal target pattern (sequence of closed curves) 
       even when the minor steps form spiral steps, 
 \item hexagonal spiral pattern even when 
       the minor steps form a closed curve.
\end{itemize}
We here propose a way to examine this situation
by our approach.

By the above situation,
we consider that the fundamental Wulff shape $\mathcal{W}_\gamma$ of 
minor steps satisfy the following:
\begin{itemize}
 \item The unit normal vector ${\bf N}_j$ of $j$-th facet
       of $\mathcal{W}_\gamma$ is
       \begin{equation}
	\label{facets: wulff ill-sp}
	{\bf N}_j = \left( \cos \frac{\pi j}{3}, \sin \frac{\pi j}{3} \right)
        \quad \mbox{for} \ 0 \le j \le 5.	
       \end{equation}
 \item The surface energy of $0$-th facet is $a$ times of the other facets,
       i.e., 
       \begin{equation}
	\label{energy: wulff ill-sp}
	\gamma ({\bf N}_0) = a \gamma ({\bf N}_j)
	 \quad \mbox{for} \ 1 \le j \le 5,
       \end{equation}
       where $0 < a < 1$ is a constant.
\end{itemize}
Note that if the Frank diagram
$\mathcal{F}_\gamma = \{ p \in \mathbb{R}^2 ; \ \gamma (p) \le 1 \}$ 
is represented as a convex hull
of its vertices ${\bf q}_j$, i.e., 
\[
 \mathcal{F}_\gamma 
 = \left\{ \sum_{j=0}^{N_\gamma - 1} c_j {\bf q}_j ; \ 
 0 \le c_j \le 1, \ \sum_{j=0}^{N_\gamma - 1} c_j {\le} 1 \right\},
\] 
then we observe that
\[
 \gamma^\circ (p)
 = \sup \{ p \cdot q ; \ \gamma (q) \le 1 \}
 = \max_{0 \le j \le N_\gamma - 1} p \cdot {\bf q}_j.
\]
Therefore, we now set the vertices ${\bf q}_j$ of $\mathcal{F}_\gamma$
as
\[
 {\bf q}_0 = \frac{1}{a} {\bf N}_0 = \left( \frac{1}{a}, 0 \right),
 \quad 
 {\bf q}_j = {\bf N}_j 
 \quad \mbox{for} \ 1 \le j \le 5
\]
to establish \eqref{facets: wulff ill-sp} and \eqref{energy: wulff ill-sp}.
To establish the successively rotating anisotropies of minor steps,
we set
\begin{equation}
 \gamma_\ell^\circ (p) = \max_{0 \le j \le 5} 
 p \cdot {\bf q}_{j + \ell} 
 \quad \mbox{for} \ 0 \le \ell \le 5.
 \label{supp for ill-sp}
\end{equation}
Here, we have considered $j + \ell \in \mathbb{Z} / (6 \mathbb{Z})$
on the above formula.
See Figure~\ref{energy for ill-sp} for the profiles of $\mathcal{W}_{\gamma_0}$ 
and $\mathcal{F}_{\gamma_0}$ for $\gamma_0$ and $\gamma_0^\circ$
by the above setting with $a = 0.5$.
\begin{figure}[htbp]
 \begin{center}
  \includegraphics[scale=0.5]{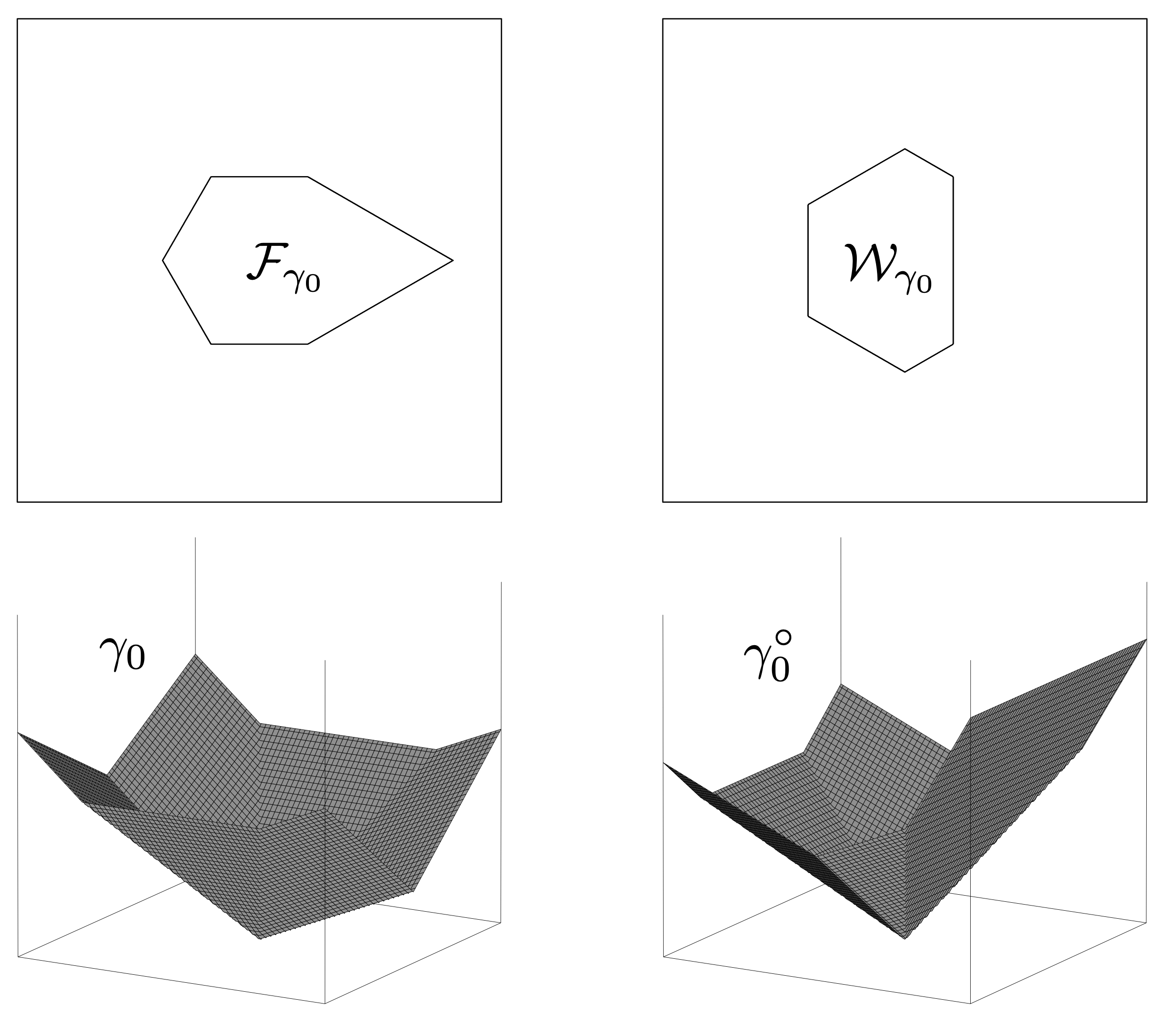}
  \caption{Graphs of $\gamma_0$ and $\mathcal{F}_{\gamma_0}$,
  or $\gamma_0^\circ$ and $\mathcal{W}_{\gamma_0}$
  defined by \eqref{supp for ill-sp} 
  with $a = 0.5$.}
  \label{energy for ill-sp}
 \end{center}
\end{figure}

Figure~\ref{ill-loop} and \ref{ill-sp} show the profiles 
of illusory loops and spirals by the above algorithm with $a=0.5$,
respectively.
The settings are as follows.
\begin{itemize}
 \item \textbf{Evolution equation:} 
       $V_{\gamma_\ell} = 3 (1 - 0.02 \kappa_{\gamma_\ell})$
       for $0 \le \ell \le 5$.
 \item \textbf{Centers:}  
	    \begin{itemize}
	     \item (Figure~\ref{ill-loop})
		   $N=1$, $a_1=(0,0)$,  
		   $m_1 = 6$.
	     \item (Figure~\ref{ill-sp})
		   $N=2$, $a_1 = (-0.1, 0)$, $a_2 = (0.1, 0)$, 
		   $m_1 = 6$, $m_2 = -6$.
	    \end{itemize}
 \item \textbf{Initial curve:}
       $\Sigma_0$ is the curve given by 
       setting $u_0 \equiv - \pi$.
\end{itemize}
In both cases, there are slower steps, catching
the following faster steps, 
and then bunching phenomena make an irregular pattern
against the situation of minor steps.
In the case of Figure~\ref{ill-loop}, 
each minor step forms a spiral curve that is not closed.
However, bunches of minor steps make a target pattern
which looks like a sequence of closed loops.
Similarly, in the case of Figure~\ref{ill-sp},
a bunch of minor steps make a spiral pattern
although each minor step forms a closed curve.

\begin{figure}[htbp]
 \begin{center}
  \includegraphics[scale=0.8]{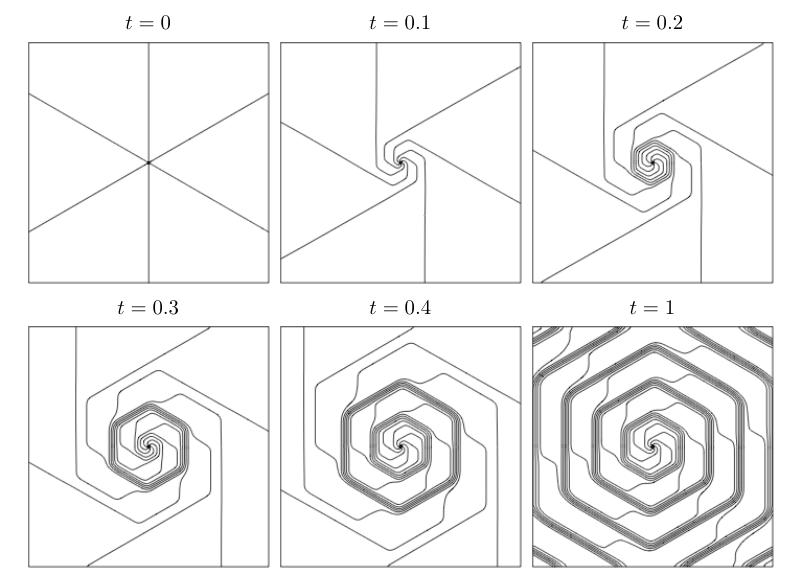}
  \caption{Numerical simulation forming
  illusory loops.}
  \label{ill-loop}
 \end{center}
\end{figure}

\begin{figure}[htbp]
 \begin{center}
  \includegraphics[scale=0.8]{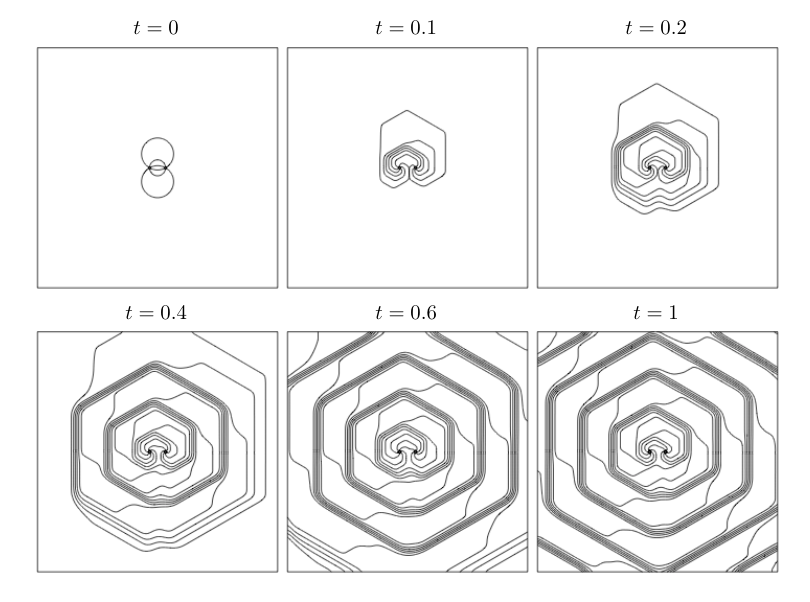}
  \caption{Numerical simulation forming
  illusory spirals.}
  \label{ill-sp}
 \end{center}
\end{figure}

\section{Summary}

In this paper, we have introduced a minimizing movements algorithm 
based on  \cite{Chambolle:2004} and \cite{OOTT:2011CMS}
for the evolution of spirals by crystalline curvature flow.
Chambolle's formulation \cite{Chambolle:2004} involves 
a signed distance function to the interface. 
However, in the case of the evolving spirals, the signed distance function 
of a spiral curve is not convenient to work with, 
even in a small neighborhood around the spiral.
Therefore, this paper introduced a modification 
that does not require the construction of 
signed distance functions after the time step.
This modification, therefore, leads to a simpler numerical algorithm.
As in the minimizing movements formulation, 
the algorithm requires solving optimization problems
involving $L^1$-type surface energy density.
We proved the existence of minimizers for our approach and showed that 
the derivatives of the minimizers 
(the functions whose zero level sets embed the spirals) 
are in $BV$.

For the evolution of a single spiral, 
our computational results are 
very close to those computed by 
the front-tracking algorithm 
developed by \cite{IO:DCDS-B}.
Since our numerical algorithm for advancing the spiral involves outer and inner loops,
we present empirical evidence that the number of inner and outer loops 
does not increase as the underlying Cartesian grid is refined. 

By using level set functions, our approach enables us 
to handle merging of multiple spirals. 
Moreover, since our approach does not rely on the distance functions to the spirals,
it allows us to handle the case 
in which bunching occurs.
As applications of it, we showcased computational results on interlace motion
\cite{Verma:PhylosMag1951, vanEnckevort:2004ActaCry},
and illusory loops and spirals \cite{Shtukenberg:2013}.

%
%

\section*{Acknowledgment}

The first author is partly supported by JSPS KAKENHI Grant Number 21K03319.
Tsai is partially supported by NSF Grant DMS-2110895.
Tsai also acknowledges support from the National Center for Theoretical Sciences (NCTS), Taiwan, during his visits.

\section*{Appendix}

\subsection*{A front tracking model for a single spiral}

In this section we present a quick review on
the model by \cite{IO:DCDS-B}.

Assume that $\mathcal{W}_\gamma$ is a convex polygon 
having $N_\gamma$ facets.
The normal and tangential vector
of $j$-th facet of $\mathcal{W}_\gamma$ is
denoted by 
\[
 {\bf N}_j = (\cos \tilde{\vartheta}_j, \sin \tilde{\vartheta}_j), \quad
 {\bf T}_j = (\sin \tilde{\vartheta}_j, - \cos \tilde{\vartheta}_j).
\]
We also denote the length of the $j$-th facet as $\ell_j > 0$.
We here assume that
$\tilde{\vartheta}_j$ satisfies \eqref{polygonal angle}--\eqref{minimal condi},
and extend $\tilde{\vartheta}_j$ to those for $j \in \mathbb{Z}$ by
\[
 \tilde{\vartheta}_{j+n N_\gamma} = \tilde{\vartheta}_j + 2 \pi n
 \quad \mbox{for} \ j = 0, 1, \ldots, N_\gamma - 1.
\]
We choose $\tilde{\vartheta}_0 \in [-\pi, \pi)$ in this section.
Accordingly,
the facets of $\mathcal{W}_\gamma$ are numbered with counter-clockwise rotation,
and 
the numbering of facets are extended 
from $j= 0, 1, \ldots, N_\gamma - 1$ to $j \in \mathbb{Z} / (N_\gamma \mathbb{Z})$
as
\[
 {\bf N}_{j+n N_\gamma} = {\bf N}_j, \quad
 {\bf T}_{j+n N_\gamma} = {\bf T}_j, \quad
 \ell_{j+n N_\gamma} = \ell_j
\] 
for every $n \in \mathbb{Z}$ and $j = 0, 1, \ldots, N_\gamma-1$.
See Figure~\ref{notation wulff diag} for details of
the notations for $\mathcal{W}_\gamma$.
\begin{figure}[htbp]
 \begin{center}
  \includegraphics[scale=0.25]{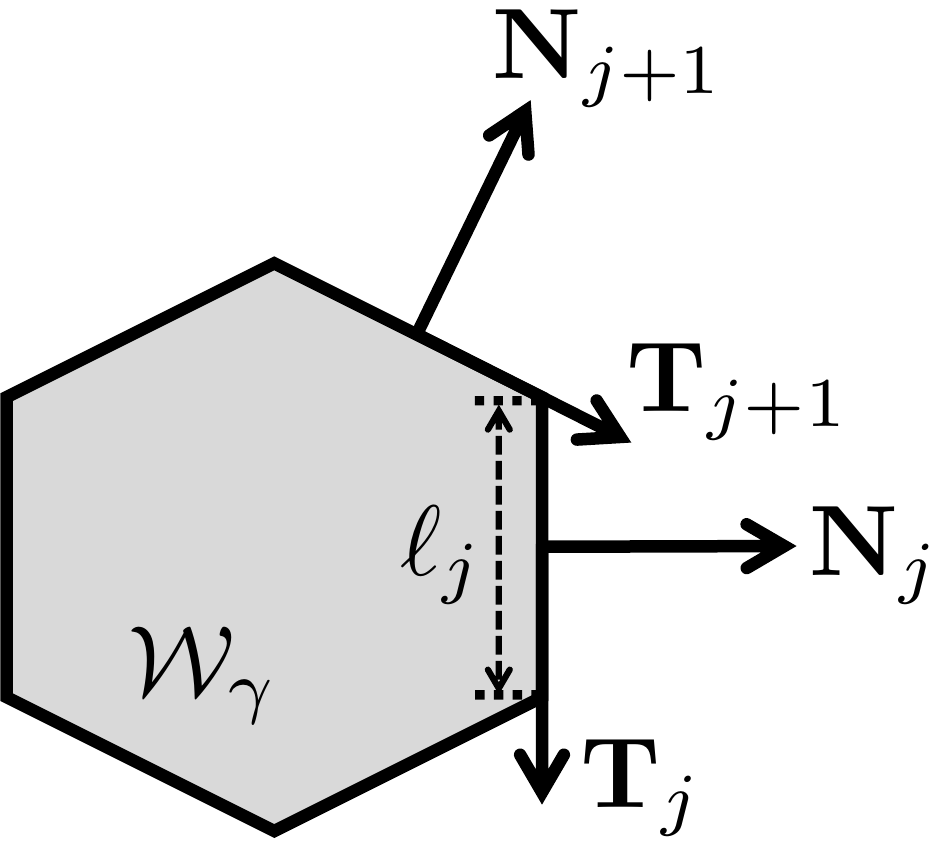}
  \caption{Notations for the Wulff diagram $\mathcal{W}_\gamma$.}
  \label{notation wulff diag}
 \end{center}
\end{figure}

Let $\Sigma_{\textup{d}} (t) = \bigcup_{j=0}^k \Sigma_{\textup{d},j} (t)$
be a polygonal spiral {evolving by the front-tracking model of \cite{IO:DCDS-B}}.
We denote the $j$-th facet of $\Sigma_{\textup{d}} (t)$
by $\Sigma_{\textup{d}, j} (t)$.
Here, we consider only the case when $\Sigma_{\textup{d}} (t)$ is convex, {so that} we assume that $\Sigma_{\textup{d}, j} (t)$
is parallel to the $j$-th facet of $\mathcal{W}_\gamma$ for 
$j \in \mathbb{Z} / (N_\gamma \mathbb{Z})$.
We also denote the center of $\Sigma_{\textup{d}} (t)$
by $y_k (t)$, and the vertices of $\Sigma_{\textup{d}} (t)$
by $\{ y_j (t) | \ j=0, 1, 2, \ldots, k \}$.
Assume that $y_k (t) = O$, i.e., 
the center is fixed at the origin.
We also assume that  
$\Sigma_{\textup{d}, 0} (t)$ has infinite length.
Then, $\Sigma_{\textup{d}, j} (t)$ is described as
\begin{align}
 \label{discrete facets}
 \Sigma_{\textup{d},j} (t) 
 & = \left\{
 \begin{array}{ll}
  \{ y_j (t) + \sigma {\bf T}_j ; \ \sigma \in [0,d_j (t)] \} & 
    \mbox{if} \ j=1, 2, \ldots, k, \\
  \{ y_0 (t) + \sigma {\bf T}_0 ; \ \sigma \in [0, \infty) \} &
   \mbox{if} \ j=0, 
 \end{array}
 \right. \\
 \label{discrete vertices}
 y_{j-1} (t) & = y_j (t) + d_j (t) {\bf T}_j 
 \quad \mbox{for} \ j=k, k-1, k-2, \ldots, 1, \\
 \label{discrete origin}
 y_k (t) & = O,
\end{align}
where $d_j (t)$ denotes the length of $\Sigma_{\textup{d}, j} (t)$.
By the above formulation and
\cite{AngenentGurtin:1989ARMA, Taylor:1991},
the crystalline curvature of $\Sigma_{\textup{d}, j} (t)$
is defined by
\[
 \kappa_\gamma = \frac{\ell_j}{d_j (t)}.
\]
Thus, \eqref{geo mcf} is translated as 
a formula of normal velocity $V_j$ of $j$-th facet
of the form
\[
 V_j = - \frac{\ell_j}{d_j} + f.
\]
Consequently, the motion of $\Sigma_{\textup{d}} (t)$ by \eqref{geo mcf}
is described by the evolution of $d_j (t)$
by the following ODE system;
\begin{align}
 \label{crystalline ODE}
 \left\{
 \begin{aligned}
  \dot{d}_k 
  & = c_k^- \left( f - \frac{\ell_{k-1}}{d_{k-1}} \right), \\
  \dot{d}_{k-1} 
  & = - b_{k-1} \left( f - \frac{\ell_{k-1}}{d_{k-1}} \right) 
  + c_{k-1}^- \left( f - \frac{\ell_{k-2}}{d_{k-2}} \right), \\
  \dot{d}_j 
  & = - b_j \left( f - \frac{\ell_j}{d_j} \right) 
  + c_j^+ \left( f - \frac{\ell_{j+1}}{d_{j+1}} \right)
  + c_j^- \left( f - \frac{\ell_{j-1}}{d_{j-1}} \right) \\
  & \hspace{7cm} \mbox{for} \ j=2, 3, \ldots, k-2, \\
  \dot{d}_1 
  & = - b_1 \left( f - \frac{\ell_1}{d_1} \right) 
  + c_1^+ \left( f - \frac{\ell_2}{d_2} \right)
  + c_1^- f,
 \end{aligned}
 \right.
\end{align}
where $b_j$ and $c_j^\pm > 0$ are constants
determined by $\tilde{\vartheta}_j$. 
Moreover, 
the generation of a new facet at the origin is proposed
in \cite{IO:DCDS-B} for the evolution of a spiral;
as the result of the motion of $\Sigma_{\textup{d}} (t)$,
if the length of $\Sigma_{\textup{d}, k} (t)$ reaches 
the critical length $\ell_k / f$,
then we add a new facet $\Sigma_{\textup{d}, k+1} (\cdot)$ 
to $\Sigma_{\textup{d}}$
and change the fixed center of $\Sigma_{\textup{d}} (t)$ to $y_{k+1} (t)$
at $t = T_k$.
In summary, an algorithm by \cite{IO:DCDS-B} 
is outlined as follows.
\begin{enumerate}
 \item Give an initial data $\Sigma_{\textup{d}} (0) 
       = \bigcup_{j=0}^k \Sigma_{\textup{d}, j} (0)$
       with suitable initial data $d_j (0)$ for $j=1, \ldots, k-1$
       and $d_{k} (0) = 0$.
       Set $T_{k-1} = 0$.
 \item \label{solve ode system}
       Solve \eqref{crystalline ODE} in $(T_{k-1}, \infty)$.
 \item Let us set $T_k = \sup \{ T > T_{k-1} | \ d_k (t) < \ell_k / f 
       \ \mbox{for} \ t \in [T_{k-1}, T) \}$.
       If $T_k < \infty$, then 
       we add a new facet $\Sigma_{\textup{d}, k+1} (T_k)$ 
       to $\Sigma_{\textup{d}} (T_k)$,
       and change the fixed center of $\Sigma_{\textup{d}} (t)$ to 
       $y_{k+1} (t)$ at $t = T_k$.
 \item Return to \eqref{solve ode system}
       with changing its center facet number $k$ to $k+1$.
\end{enumerate}
Note that the evolution law, in particular 
the behavior of the terminal facets $\Sigma_{\textup{d}, k} (t)$
and $\Sigma_{\textup{d}, 0} (t)$
of the front-tracking model or the boundary condition
are slightly different from our minimizing movements apporach.
However, in \S \ref{sec: unit spiral} we have seen that 
$\Sigma_{\textup{d}} (t)$ and {$\Sigma_{h} (t)$} are numerically close.

\subsection*{Proof of Lemma \ref{lem: aniso total variation}}

We now prove 
\begin{enumerate}
 \item $J_\gamma$ is convex and nonnegative,
 \item $J_\gamma (w) < \infty$ for $w \in BV(W)$, 
 \item $\liminf_{u \to v} J_\gamma (u) \ge J_\gamma (v)$
       in $L^1 (W)$,
 \item If $w \in W^{1,1} (W)$, then
       $J_\gamma (w) = \int_W \gamma (\nabla (w - \theta)) dx$.
\end{enumerate}
However, we omit the proof of 
\eqref{J: convex and positive}
since it is clear.
See \eqref{aniso total variation} for the definition of $J_\gamma$.

\begin{proof}

 We demonstrate 
 \eqref{J: well-defined} and \eqref{J: lower semiconti}. 
 Let $w \in BV(W)$
 and $\varphi \in C^1_c (W;\mathbb{R}^2)$ with $\gamma^\circ (\varphi) \le 1$.
 Note that
 \begin{align}
  & \label{support: positivity}
   \frac{1}{\Lambda_\gamma} \le \gamma^\circ \le \Lambda_\gamma \quad
   \mbox{on} \ \mathbb{S}^1, \\
  & 
  \gamma (p) \le 1, \ \mbox{or} \ \gamma^\circ (p) \le 1
  \quad \mbox{implies} \quad |p| \le \Lambda_\gamma
  \label{bound of what satisfies gamma < 1}
 \end{align}
 by \eqref{gamma: homogeneity} and \eqref{gamma: positivity}.
 Then, 
 \begin{align*}
  & - \int_W w \mathrm{div} \varphi dx - \int_W \nabla \theta \cdot \varphi dx \\
  & \le \| \varphi \|_\infty \left( - \int_W w \mathrm{div} \frac{\varphi}{\| \varphi \|_\infty} dx \right) 
  + |W| \|\nabla  \theta \|_\infty \| \varphi \|_\infty \\
  & \le \Lambda_\gamma ([w]_{BV} + |W| \| \nabla \theta \|_\infty )
 \end{align*}
 by \eqref{bound of what satisfies gamma < 1},
 where $|W|$ is a Lebesgue measure of $W$, 
 and $\| \varphi \|_\infty = \sup_W | \varphi |$ 
 for $\varphi \in C (\overline{W}; \mathbb{R}^2)$.
 It implies \eqref{J: well-defined}.
 
 Next, 
 we observe that
 \[
 \int_W u \mathrm{div} \varphi dx \to \int_W v \mathrm{div} \varphi dx
 \quad \mbox{as} \ u \to v \ \mbox{in} \ L^1 (W)
 \]
 since $\mathrm{div} \varphi$ is bounded on $\overline{W}$.
 Hence, we obtain
 \begin{align*}
  \liminf_{u \to v} J_\gamma (u)
  & \ge \liminf_{u \to v}
  \left[ - \int_W u \mathrm{div} \varphi dx - \int_W \nabla \theta \cdot \varphi dx \right] \\
  & = - \int_W v \mathrm{div} \varphi dx - \int_W \nabla \theta \cdot \varphi dx,
 \end{align*}
 which implies \eqref{J: lower semiconti}.

 Finally, we prove \eqref{J: aniso peri}.
 First, let $w \in W^{1,1} (W)$ and
 $\varphi \in C^1_c (W; \mathbb{R}^2)$ 
 satisfying $\gamma^\circ(\varphi) \le 1$.
 Then, we obtain
 \begin{align*}
  - \int_W w \mathrm{div} \varphi dx
  - \int_W \nabla \theta \cdot \varphi dx   
  = \int_W \nabla (w - \theta) \cdot \varphi dx
  \le \int_W \gamma(\nabla(w - \theta)) dx
 \end{align*}
 since $\gamma (p) = \sup \{ p \cdot q ; \ \gamma^\circ(q) \le 1 \}$,
 which implies $J_\gamma (w) \le \int_W \gamma (\nabla (w - \theta)) dx$.

 We next show that
 \begin{equation}
  \int_W \gamma (\nabla (w - \theta)) dx \le J_\gamma (w)  
  \label{iv: aniso-peri ineq other}
 \end{equation}
 for $w \in W^{1,1} (W)$.
 We first note that $\gamma$ is Lipschitz continuous,
 i.e., 
 \begin{equation}
  \label{gamma: Lipschitz conti}
  |\gamma (p_1) - \gamma (p_2)| \le \Lambda_\gamma |p_1 - p_2|
  \quad \mbox{for} \ p_1, p_2 \in \mathbb{R}^2.  
 \end{equation}
 This implies that 
 \begin{align*}
  \left|
  \int_W \gamma (\nabla (w - \theta)) dx 
  - \int_W \gamma (\nabla (v - \theta)) dx
  \right|
  \le \Lambda_\gamma \| w - v \|_{W^{1,1}} 
  \qquad \\
  \mbox{for} \ w, v \in W^{1,1} (W).  
 \end{align*}
 Then, 
 it suffices to prove \eqref{iv: aniso-peri ineq other} for $w \in C^1_c (\mathbb{R}^2)$,
 since $C^1_c (\mathbb{R}^2)$ is dense in $W^{1,1} (W)$.

 We now make some preparations for the proof.
 Note that $\nabla (w - \theta)$ is uniformly continuous on $\overline{W}$.
 Let us fix $\varepsilon > 0$ arbitrary,
 and fix $\delta > 0$ satisfying
 \[
  \sup_{|y-x| < \delta} |\nabla (w - \theta) (y) - \nabla (w - \theta) (x)|
  < \frac{\varepsilon}{4 \Lambda_\gamma |W|}.
 \]
 Then, since $\{ B_{\delta} (x) ; \ x \in W \}$ is an open covering
 of $\overline{W}$,
 there exists $x_1, x_2, \ldots, x_M \in W$ such that
 $\overline{W} \subset \bigcup_{j=1}^M B_{\delta} (x_h)$.
 Let $\chi_j \colon \mathbb{R}^2 \to [0,1]$ be the partition of unity,
 i.e., $\chi_j \in C^\infty (\mathbb{R}^2)$ for $j=0,1, \ldots, M$ such that
 \begin{itemize}
  \item $\mathrm{supp} \chi_j \subset B_{\delta} (x_j)$ for $j=1, \ldots, M$,
	and $\mathrm{supp} \chi_0 \subset \mathbb{R}^2 \setminus \overline{W}$,
  \item $\sum_{j=0}^M \chi_j \equiv 1$,
	in particular $\sum_{j=1}^M \chi_j (x) = 1$ for $x \in \overline{W}$.
 \end{itemize}
 Let $q_{\varepsilon, j} \in \mathbb{R}^2$ be such that
 $\gamma^\circ (q_{\varepsilon, j}) \le 1$ and
 \[
  \gamma (\nabla (w - \theta) (x_j)) - \frac{\varepsilon}{2|W|} 
  < \nabla (w - \theta) (x_j) \cdot q_{\varepsilon, j}
  \quad \mbox{for} \ j=1, \ldots, M.
 \]
 Let $d(x)$ be a distance function of $\partial W$
 for $x \in \mathbb{R}^2$, i.e., 
 \[
  d(x) = \left\{
  \begin{array}{rl}
   \inf \{ |x-y| ; \ y \in \partial W \} & \mbox{if} \ x \in W, \\
   - \inf \{ |x-y| ; \ y \in \partial W \} & \mbox{otherwise}. \\
  \end{array}
  \right.
 \]
 Then, since $\partial W$ is smooth, 
 there exists $\delta_0 > 0$ such that
 $d \in C^1 (\partial W^{\delta_0})$,
 where 
 \[
  \partial W^{\delta_0} = \{
  x \in \mathbb{R}^2 ; \ |d(x)| < \delta_0 \}.
 \]
 We now consider a cut-off function $\zeta \in C^1 (\mathbb{R})$ 
 such that
 $\zeta$ is monotone nondecreasing and satisfies
 \[
  \zeta (r) = 
  \left\{
  \begin{array}{ll}
   0 & \mbox{if} \ r \le 1/2, \\
   1 & \mbox{if} \ r \ge 1.
  \end{array}
  \right.
 \]
 Then, we observe that 
 $\tilde{d}_n (x) = \zeta (n d(x)) \in C^1 (\mathbb{R}^2)$
 for sufficiently large $n \in \mathbb{N}$.

 Set 
 \[
  \mathcal{J} (w; \varphi) = - \int_W w(y) \mathrm{div} \varphi (y) dy
  - \int_W \nabla \theta (y) \cdot \varphi (y) dy
 \]
 for $w \in C^1_c (W)$ and 
 $\varphi \in C^1_c (\overline{W}; \mathbb{R}^2)$.
 We now choose $n \in \mathbb{N}$ sufficiently large so that
 $\tilde{d}_n (x) \in C^1 (\mathbb{R}^2)$,
 $\tilde{d}_n (x_j) = 1$ 
 for every $j=1, \ldots, M$, and 
 \begin{align*}
  | \partial W^{\frac{1}{n}} \cap W | < \frac{\varepsilon}{\Lambda_\gamma (1 + L)}, 
  \quad \mbox{where} \ L = \sup_{\overline{W}} |\nabla (w - \theta)|.
 \end{align*}
 Note that
 \[
  W \cap \partial W^{\frac{1}{n}}
  = \{ x \in \overline{W} ; \ 0 < \tilde{d}_n (x) < 1 \},
  \quad
  W \setminus \partial W^{\frac{1}{n}} 
  = \{ x \in \overline{W} ; \ \tilde{d}_n (x) = 1 \}.
 \]
 Now, let us set
 $\varphi_\varepsilon (x) = \sum_{j=1}^M \tilde{d}_n (x) \chi_j (x) q_{\varepsilon, j}$.
 Then, $\varphi_\varepsilon \in 
 C^1_c (\overline{W}; \mathbb{R}^2)$
 and
 \[
  \gamma^\circ (\varphi_\varepsilon (x))
  \le \sum_{j=1}^M \chi_j (x) = 1
 \]
 for $x \in \overline{W}$
 by \eqref{gamma: convexity} and \eqref{gamma: homogeneity}.
 We now calculate $\mathcal{J} (w;\varphi_\varepsilon)$
 with integration by parts and obtain
 \begin{align}
  & \int_W \gamma (\nabla (w - \theta) (y)) dy
  - \mathcal{J} (w; \varphi_\varepsilon) 
  \nonumber \\
  & = \int_W \Bigl( \gamma (\nabla (w - \theta) (y))
  - \nabla(w - \theta) (y) \cdot \varphi_\varepsilon (y) \Bigr) dy
  = I + II,
  \label{aniso-peri estimate integral}
 \end{align}
 where
 \begin{align*}
  I & = \int_{W_1} \Bigl( \gamma (\nabla (w - \theta) (y))
  - \nabla(w - \theta) (y) \cdot \varphi_\varepsilon (y) \Bigr) dy, 
  \quad W_1 = W \setminus \partial W^{\frac{1}{n}}, \\
  II & = \int_{W_2} \Bigl( \gamma (\nabla (w - \theta) (y))
  - \nabla(w - \theta) (y) \cdot \varphi_\varepsilon (y) \Bigr) dy,
  \quad W_2 = W \cap \partial W^{\frac{1}{n}}.
 \end{align*}

 In the first part, 
 since $\tilde{d}_n = 1$ on $W_1$ we have
 \begin{align*}
  I = \sum_{j=1}^M 
  \int_{W_1 \cap B_\delta (x_j)}  
  \chi_j (y) \Bigl( \gamma (\nabla (w - \theta) (y))
  - \nabla(w - \theta) (y) \cdot q_{\varepsilon, j} \Bigr) dy.
 \end{align*}
 For $y \in W_1 \cap B_{\delta} (x_j)$
 we observe that
 \begin{align*}
  & \gamma (\nabla (w - \theta) (y)) 
  - \nabla(w - \theta) (y) \cdot q_{\varepsilon, j} \\
  & \le 
  \gamma (\nabla (w - \theta) (x_j)) 
  - \nabla(w - \theta) (x_j) \cdot q_{\varepsilon, j} 
  \\
  & \qquad 
  + (\Lambda_\gamma + |q_{\varepsilon, j}|) 
  |\nabla (w - \theta)(y) - \nabla (w - \theta)(x_j)| \\
  & < 
  \frac{\varepsilon}{|W|}
 \end{align*}
 by \eqref{gamma: Lipschitz conti} and \eqref{bound of what satisfies gamma < 1}.
 Then, we obtain
 \begin{align}
  I
  < \frac{\varepsilon}{|W|} \sum_{j=1}^M \int_{W_1 \cap B_\delta (x_j)} \chi_j (y) dy 
  = \varepsilon.
  \label{aniso-peri estimate I}
 \end{align}
 On the second part of \eqref{aniso-peri estimate integral}, 
 we give an estimate of
 \begin{align*}
  II = \sum_{j=1}^M 
  \int_{W_2 \cap B_\delta (x_j)}  
  \chi_j (y) \Bigl( \gamma (\nabla (w - \theta) (y))
  - \tilde{d}_n (y) \nabla(w - \theta) (y) \cdot q_{\varepsilon, j} \Bigr) dy.
 \end{align*}
 Note that $\gamma (p) \le \Lambda_\gamma |p|$ for every $p \in \mathbb{R}^2$.
 Thus we have
 \begin{align*}
  & \left| \Bigl( \gamma (\nabla (w - \theta) (y))
  - \tilde{d}_n (y) \nabla(w - \theta) (y) \cdot q_{\varepsilon, j} \Bigr) \right| \\
  & \le 
  \gamma (\nabla (w - \theta) (y))
  + |\nabla(w - \theta) (y)| |q_{\varepsilon, j}| \\
  & \le 2 \Lambda_\gamma L
 \end{align*}
 for $y \in W_2 \cap B_\delta (x_j)$.
 Hence, we obtain
 \begin{align}
  |II| \le 2 \Lambda_\gamma L \sum_{j=1}^M 
  \int_{W_2 \cap B_\delta (x_j)} \chi_j (y) dy
  = 2 \Lambda_\gamma L |W_2| 
  \le 2 \varepsilon.
  \label{aniso-peri estimate II}
 \end{align}
 By combining \eqref{aniso-peri estimate integral}, 
 \eqref{aniso-peri estimate I} and 
 \eqref{aniso-peri estimate II},
 we obtain
 \begin{align*}
  \int_W \gamma (\nabla (w - \theta) (y)) dy
  - \mathcal{J} (w; \varphi_\varepsilon)
  \le 3 \varepsilon,
 \end{align*}
 which implies
 \[
  \int_W \gamma (\nabla (w - \theta)) dx
  \le J_\gamma (w) + 3 \varepsilon.
 \]
 By tending $\varepsilon \to 0$ we obtain \eqref{iv: aniso-peri ineq other}
 for $w \in C^1_c (\mathbb{R}^2)$, and thus \eqref{J: aniso peri}.
\end{proof}

\bibliography{min_move}
\bibliographystyle{plain}

\end{document}